\documentclass{article}
\usepackage[utf8]{inputenc}
\usepackage{amsmath, amsthm, amsfonts, amssymb, bbm}
\usepackage{geometry, cases}
\usepackage{hyperref} 
\usepackage[shortlabels]{enumitem}

\usepackage{graphicx, subcaption}
\usepackage[usenames,dvipsnames]{color}

\newcommand{\eps}{\varepsilon}

\renewcommand{\Pr}{\mathbb{P}}

\newcommand{\bN}{\mathbb{N}}
\newcommand{\E}{\mathbb{E}}

\newcommand{\cA}{\mathcal{A}}
\newcommand{\cB}{\mathcal{B}}

\newcommand{\cG}{\mathcal{G}}
\newcommand{\cE}{\mathcal{E}}
\newcommand{\cF}{\mathcal{F}}

\newcommand{\cT}{\mathcal{T}}

\newcommand{\TV}{\mathrm{d}_{\mathrm{TV}}}
\DeclareMathOperator{\var}{Var}

\newcommand{\T}{\mathcal{U}_{n,m}}

\newcommand\lrpar[1]{\left(#1\right)}
\newcommand{\epsp}{\zeta} 

\newcommand\abs[1]{\left|#1\right|} 
\newcommand\floor[1]{\left\lfloor #1 \right \rfloor}
\newcommand\ceil[1]{\left\lceil #1 \right \rceil}

\newcommand{\xone}{x^{(1)}}
\newcommand{\xonetwo}{x^{(i)}}
\newcommand{\xtwo}{x^{(2)}}
\newcommand{\xot}{x^{(0)}}
\newcommand{\Xonetwo}{X^{(i)}}
\newcommand{\Xone}{X^{(1)}}
\newcommand{\Xtwo}{X^{(2)}}
\newcommand{\ionetwo}{I^{(i)}}
\newcommand{\ione}{I^{(1)}}
\newcommand{\itwo}{I^{(2)}}

\newcommand{\nN}{n_N}

\newcommand{\eve}{\mathcal{B}}

\newcommand\Var{\operatorname{Var}} 

\newcommand{\isub}{\sqsubseteq}

\newcommand{\Aut}{\mathrm{Aut}}

\renewcommand{\u}{u}
\renewcommand{\v}{v}

\newcommand{\w}{w}

\newcommand{\ps}{\hat{p}}
\newcommand{\pz}{p^*}

\DeclareMathOperator{\pr}{\mathbb{P}}

\newtheorem{thm}{Theorem}

\newtheorem{lemma}[thm]{Lemma}
\newtheorem{cor}[thm]{Corollary}

\newtheorem{remark}{Remark}
\newtheorem{problem}{Problem}

\newtheorem*{lemma*}{Lemma}

\newcommand\xpar[1]{(#1)}
\newcommand\bigpar[1]{\bigl(#1\bigr)}
\newcommand\Bigpar[1]{\Bigl(#1\Bigr)}
\newcommand\biggpar[1]{\biggl(#1\biggr)}
\newcommand\Biggpar[1]{\Biggl(#1\Biggr)}
\newcommand\bigsqpar[1]{\bigl[#1\bigr]}
\newcommand\Bigsqpar[1]{\Bigl[#1\Bigr]}
\newcommand\biggsqpar[1]{\biggl[#1\biggr]}
\newcommand\Biggsqpar[1]{\Biggl[#1\Biggr]}
\newcommand\lrsqpar[1]{\left[#1\right]}

\newcommand\bigcpar[1]{\bigl\{#1\bigr\}}
\newcommand\Bigcpar[1]{\Bigl\{#1\Bigr\}}
\newcommand\biggcpar[1]{\biggl\{#1\biggr\}}
\newcommand\Biggcpar[1]{\Biggl\{#1\Biggr\}}

\newcommand\bigabs[1]{\bigl|#1\bigr|}
\newcommand\Bigabs[1]{\Bigl|#1\Bigr|}
\newcommand\biggabs[1]{\biggl|#1\biggr|}

\newcommand\bigfloor[1]{\bigl\lfloor#1\bigr\rfloor}

\newcommand{\indic}[1]{\mathbbm{1}_{\{{#1}\}}}

\newcommand{\refT}[1]{Theorem~\ref{#1}}

\newcommand{\refL}[1]{Lemma~\ref{#1}}
\newcommand{\refR}[1]{Remark~\ref{#1}}
\newcommand{\refS}[1]{Section~\ref{#1}}
\newcommand{\refA}[1]{Appendix~\ref{#1}}

\newcommand\dto{\overset{\mathrm{d}}{\to}}

\newcommand\Po{\mathrm{Po}}

\newcommand\Nor{\mathrm{N}}
\newcommand\TNor{\mathrm{SN}}

\newcommand\noproof{\qed}

\usepackage{thmtools}
\usepackage{thm-restate}

\newenvironment{romenumerate}{
\vspace{-0.25em}\begin{enumerate}
\itemsep0pt \parskip0pt \parsep0pt%
 }{\vspace{-0.05em}\end{enumerate}\vspace{-0.05em}}
 
 \makeatletter
 \def\@currentlabel{(ii)}\label{enum:contain:main} 
 \def\@currentlabel{(ii)}\label{enum:contain:dist:other} 
 \makeatother

\let\OLDthebibliography\thebibliography
\renewcommand\thebibliography[1]{
  \OLDthebibliography{#1}
  \setlength{\parskip}{0pt}
  \setlength{\itemsep}{0pt plus 0.3ex}
}

\title{Isomorphisms between dense random graphs}
\author{Erlang Surya\thanks{Department of Mathematics, University of California at San Diego, La Jolla, CA 92093, USA. Email: {\tt esurya@ucsd.edu}. Supported by NSF~CAREER grant~DMS-1945481.}
\and
Lutz Warnke\thanks{Department of Mathematics, University of California, San Diego, La Jolla, CA~92093, USA. 
E-mail: {\tt lwarnke@ucsd.edu}. 
Supported by NSF~CAREER grant~DMS-1945481, and a Sloan Research Fellowship.}
\and
Emily Zhu\thanks{Department of Mathematics, University of California, San Diego, La Jolla, CA~92093, USA. 
E-mail: {\tt e9zhu@ucsd.edu}. 
Supported by NSF~Graduate Research Fellowship Program grant~DGE-2038238, and a Sloan Research Fellowship.}}

\date{June~13, 2023; revised January~28, 2025}

\begin{document}

\maketitle

\begin{abstract}
We consider two variants of the induced subgraph isomorphism problem for two independent binomial random graphs with constant edge-probabilities~$p_1,p_2$.  
In particular, 
(i)~we prove a sharp threshold result for the appearance of~$G_{n,p_1}$ as an induced subgraph of~$G_{N,p_2}$, 
(ii)~we show two-point concentration of the size of the maximum common induced subgraph of~$G_{N, p_1}$ and~$G_{N,p_2}$,
and 
(iii)~we show that the number of induced copies of~$G_{n,p_1}$ in~$G_{N,p_2}$ has an unusual limiting~distribution.

These results confirm simulation-based predictions of McCreesh, Prosser, Solnon and~Trimble, and resolve several open problems of Chatterjee and~Diaconis.
The proofs are based on careful refinements of the first and second moment method, 
using extra twists to 
(a)~take some non-standard behaviors into account, and 
(b)~work around the large variance issues that prevent standard applications of these~methods. 
\end{abstract}

\section{Introduction}
%
Applied benchmark tests for the famous `subgraph isomorphism problem' 
empirically discovered interesting phase transitions in random graphs. 
More concretely, these phase transitions were observed in two induced variants of the `subgraph containment problem' widely-studied in random graph theory. 
In this paper we prove that the behavior of these two new random graph problems is surprisingly rich, 
with unexpected phenomena such as 
(a)~that the form of the answer changes for constant edge-probabilities, 
(b)~that the classical second moment method fails due to large variance, and 
(c)~that an unusual 
limiting distribution~arises. 

To add more context, in many applications such as pattern recognition, computer vision, biochemistry and molecular science, 
it is a fundamental problem to determine whether an induced copy of a given graph~$F$ (or a large part of a given graph~$F$) is contained in another graph~$G$; see~\cite{Cook1994,Raymond2002,Conte2004,Damiand2011,ehrlich2011maximum,Giugno2013,Bonnici2013,mccreesh2018subgraph}. 
In this paper we consider two probabilistic variants of this problem, 
where the two graphs~$F$ and~$G$ are both independent binomial random graphs with constant edge-probabilities~$p_1,p_2 \in (0,1)$. 
Our main results are~threefold:
\vspace{-0.25em}\begin{itemize}
\itemsep 0.125em \parskip 0em  \partopsep=0pt \parsep 0em 
	\item We prove a sharp threshold result for the appearance of~$G_{n,p_1}$ as an induced subgraph of~$G_{N,p_2}$, 
	and discover that the sharpness differs between the cases~$p_2=1/2$ and~$p_2 \neq 1/2$; see~Theorem~\ref{thm:contain} and~Corollary~\ref{cor:contain}.
	\item We show that the number of induced copies of~$G_{n,p_1}$ in~$G_{N,p_2}$ has a Poisson limiting distribution for~$p_2=1/2$, and a `squashed' log-normal limiting~distribution for~$p_2 \neq 1/2$; see~Theorem~\ref{thm:contain:dist}.
	\item We show two-point concentration 
	of the maximum common induced subgraph of~$G_{N, p_1}$ and~$G_{N,p_2}$, 
	and discover that the form of the maximum size changes as we vary~$p_1,p_2 \in (0,1)$; see~Theorem~\ref{thm:maincommon}.
\vspace{-0.25em}\end{itemize}
These results confirm simulation-based phase transition predictions of McCreesh, Prosser, Solnon and~Trimble~\cite{mccreesh2016subgraph,mccreesh2018subgraph}, 
and resolve several open problems of Chatterjee and Diaconis~\cite{chatterjee2021isomorphisms}. 
%
Our proofs 
are based on careful refinements of the first and second moment method, 
using several extra twists to 
(a)~take the non-standard phenomena into account, and 
(b)~work around the large variance issues that prevent standard applications of these moment based methods, 
using in particular pseudorandom properties and multi-round exposure~arguments to tame the variance.

\pagebreak[3]
\subsection{Induced subgraph isomorphism problem for random graphs}\label{sec:intro:contain}
In the \emph{induced subgraph isomorphism problem} the objective is to determine whether~$F$ is isomorphic to an induced subgraph of~$G$, 
i.e., whether~$G$ contains an induced copy of~$F$. 
In this paper we focus on this problem for two independent binomial random graphs~$F=G_{n,p_1}$ and~$G=G_{N,p_2}$, 
with constant edge-probabilities~${p_1,p_2 \in (0,1)}$ and~${n \ll N}$ many vertices. 
The motivation here is that, in applied work on benchmark tests for this NP-hard problem, it was empirically discovered~\cite{mccreesh2016subgraph,mccreesh2018subgraph} 
that such random graphs can be used to generate algorithmically hard problem-instances, 
leading to intriguing phase transitions. 
Knuth~\cite{SJPC,chatterjee2021isomorphisms} asked for mathematical explanations of these phase transitions, 
which are illustrated in Figure~\ref{fig:heatmap150} via containment probability phase-diagram plots.
The central points of these plots for~$p_1=p_2=1/2$ were resolved by Chatterjee and Diaconis~\cite{chatterjee2021isomorphisms} (and independently by Alon~\cite{Alon2017,Alon2023}), 
who emphasized in talks that the more interesting general case seems substantially more complicated. 
In this paper we resolve the general case 
with edge-probabilities~${(p_1,p_2) \in (0,1)^2}$: 
besides explaining the phase-diagram plots in~Figure~\ref{fig:heatmap150}, 
we discover that the non-uniform case~$p_2 \neq 1/2$ gives rise to new phenomena not anticipated by earlier work, 
including a different sharpness of the phase transition and an unusual `squashed' log-normal limiting~distribution. 

\begin{figure}
    \centering
    \includegraphics[width=\textwidth]{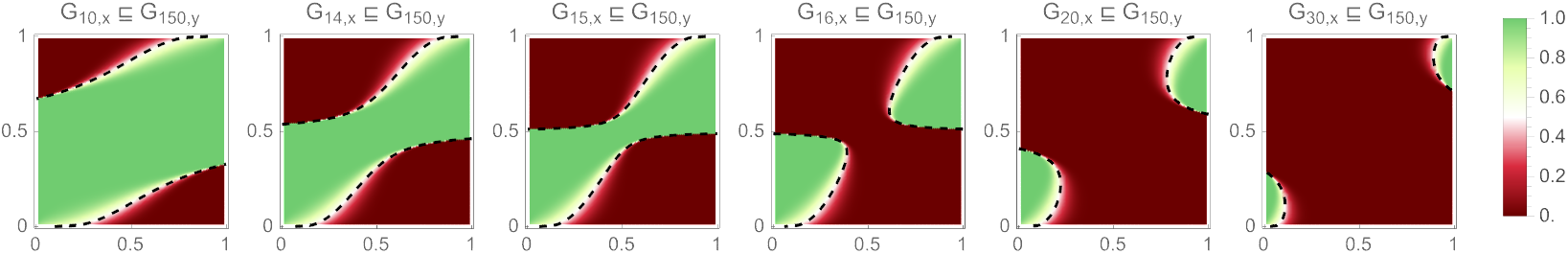}
    \caption{Theorem~\ref{thm:contain} establishes a sharp threshold around~$n^{*}=n^{*}(p_1,p_2,N):=2\log_a N + 1$ for the appearance of 
    the binomial random graph~$G_{n,p_1}$ as an induced subgraph of the independent random graph~$G_{N,p_2}$.
    It also yields the induced containment probability estimate~${\Pr(G_{n^{*}+c,p_1}\isub G_{N,p_2})\approx f_{p_1,p_2}(c)}$, which allows us to reproduce an idealized version of Figure~5 in~\cite{mccreesh2018subgraph}, where~$\Pr(G_{n,x}\isub G_{N,y})$ with~$N=150$ is empirically plotted for all~$x,y\in [0,1]$ and ${n\in \{10, 14, 15, 16, 20, 30\}}$; the dashed line corresponds to the threshold~$n^{*}$.
		Previous work~\cite{chatterjee2021isomorphisms} applied to the special case~$p_1=p_2=1/2$, i.e., only reproduced the central point in each~plot.%
\label{fig:heatmap150}}
\end{figure}

Our first result 
establishes a sharp threshold for the appearance of the binomial random graph~$G_{n,p_1}$ as an induced subgraph of the independent random graph~$G_{N,p_2}$, 
resolving an open problem of Chatterjee and Diaconis~\cite{chatterjee2021isomorphisms}. 
Below the abbreviation~$F\isub G$ means that~$G$ contains an induced copy of~$F$. 
%
\begin{thm}[Sharp threshold]\label{thm:contain}
Let~$p_1,p_2\in (0,1)$ be constants. Define $a := {1/\bigpar{p_2^{p_1}(1-p_2)^{1-p_1}}}$ and ${\varepsilon_N := (\log \log N)^2/\log N}$.
Then the following~holds, for independent binomial random graphs~$G_{n,p_1}$ and~$G_{N,p_2}$:%
\begin{romenumerate}
\item\label{enum:contain:12}
If~$p_2=1/2$, then~$a=2$ and
\begin{align}\vspace{-0.25em}
\label{eq:contain:12}
\lim_{N \to \infty}\Pr\bigpar{G_{n,p_1}\isub G_{N,p_2}} &= \begin{cases} 1 \ \ \ \  & \text{if $n \le 2\log_a N + 1-\varepsilon_N$,}\\ 0 \ \ \ \ & \text{if $n \ge 2\log_a N + 1+\varepsilon_N$.}\end{cases}\hspace{2.5em}\vspace{-0.25em}
\intertext{\item 
If~$p_2 \neq 1/2$, then} 
\label{eq:contain:main}
\lim_{N \to \infty}\Pr\bigpar{G_{n,p_1}\isub G_{N,p_2}} &= \begin{cases} 1 & \text{if $n-(2\log_a N+1) \to-\infty$,}\\ 
f(c) & \text{if $n-(2\log_a N+1)\to c \in (-\infty,+\infty)$,}\\
0 & \text{if $n-(2\log_a N+1) \to \infty$,}\end{cases}\hspace{2.5em}\vspace{-0.125em}%
\end{align}%
where~$f(c) = f_{p_1,p_2}(c):=\Pr\bigpar{\Nor(0,\sigma^2)\ge c}$ 
for a normal random variable~$\Nor(0,\sigma^2)$ with mean~$0$ and variance~$\sigma^2:= 
2p_1(1-p_1)  {\log^2_a(1/p_2-1)}$.
\end{romenumerate}\vspace{-0.125em}%
\end{thm}
%
%
In concrete words, Theorem~\ref{thm:contain} shows that, around the threshold~$n \approx 2\log_a N+1$, the induced containment probability~$\Pr(G_{n,p_1}\isub G_{N,p_2})$ drops from~$1-o(1)$ to~$o(1)$  
in a window of size at most two when~${p_2=1/2}$, whereas this window has unbounded size when~${p_2\neq 1/2}$. 
Our proof explains why this new phenomenon happens in the non-uniform case~${p_2 \neq 1/2}$:  
here the asymmetry of edges and non-edges makes~${\Pr(G_{n,p_1} \isub G_{N,p_2} \mid G_{n,p_1})}$ strongly dependent on the number of edges in~$G_{n,p_1}$, 
which does not occur in the uniform case~${p_2=1/2}$ considered in previous work~\cite{chatterjee2021isomorphisms}. 
This edge-deviation effect turns out to be the driving force behind the limiting probability ${f(c) \in (0,1)}$ in~\eqref{eq:contain:main}; 
see Section~\ref{sec:heur:ISIP} for more proof~heuristics.

Theorem~\ref{thm:contain} confirms the simulation based predictions of McCreesh, Prosser, Solnon and~Trimble~\mbox{\cite{mccreesh2016subgraph,mccreesh2018subgraph}}, 
who empirically plotted the induced containment probability~$\Pr(G_{n,x}\isub G_{N,y})$ for~$N=150$ and predicted a phase transition near~${n \approx 2\log_a N}$; see Figure~5 and Section~3.1 in~\cite{mccreesh2018subgraph}.  
Figure~\ref{fig:heatmap150} illustrates that the fuzziness they found near their predicted threshold can be explained by the limiting probability~$f(c)$ in~\eqref{eq:contain:main}, 
whose existence was not predicted in earlier work 
(of course, the `small number of vertices' effect also leads to some discrepancies in the plots, in particular for very small and large edge-probabilities in Figure~\ref{fig:heatmap150}).

The classical problems of determining the size of the largest independent set and clique of~$G_{N,p_2}$ are both related to Theorem~\ref{thm:contain}, 
as they would correspond to the excluded edge-probabilities~${p_1 \in \{0,1\}}$.
These two classical parameters~$\alpha(G_{N,p_2})$ and~$\omega(G_{N,p_2})$ are well-known~\cite{BB,JLR} to typically have size~${2 \log_{a}N - \Theta(\log \log N)}$ for~$a \in \{1/(1-p_2),1/p_2\}$, 
and this additive~$\Theta(\log \log N)$ left-shift compared to the threshold from Theorem~\ref{thm:contain} 
stems from an important conceptual difference: 
$n$-vertex cliques and independent sets have an automorphism group of size~$n!$, 
whereas~$G_{n,p_1}$ typically has a trivial automorphism group. 
This is one reason why our proof needs to take pseudorandom properties of random graphs into account; see also Sections~\ref{sec:heur:ISIP} and~\ref{sec:pseudorandom}. 


\begin{figure}
\centering
\begin{subfigure}{.29\textwidth}
\centering
\includegraphics[scale=0.25]{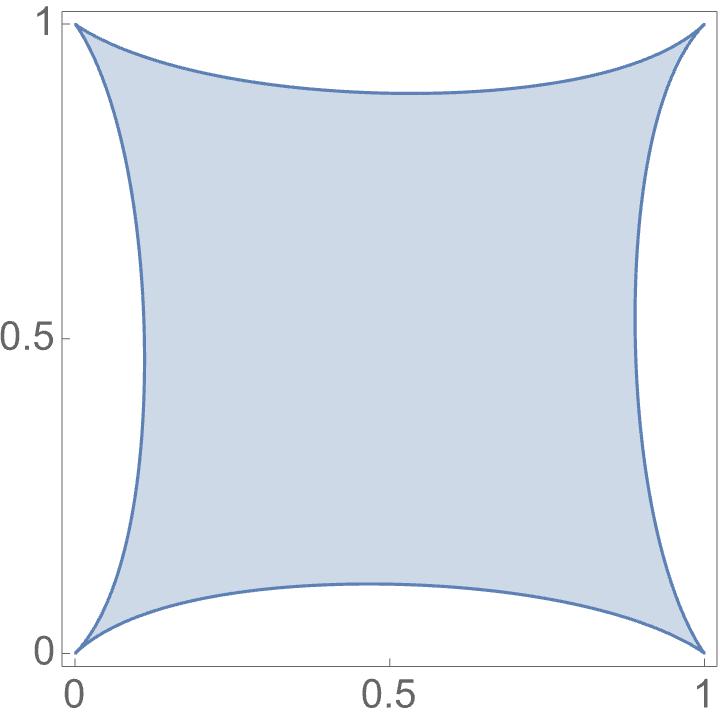}\hspace{.8em}
\caption[]{}
\label{fig:region1}
\end{subfigure}
\begin{subfigure}{.29\textwidth}
\centering
\includegraphics[scale=0.25]{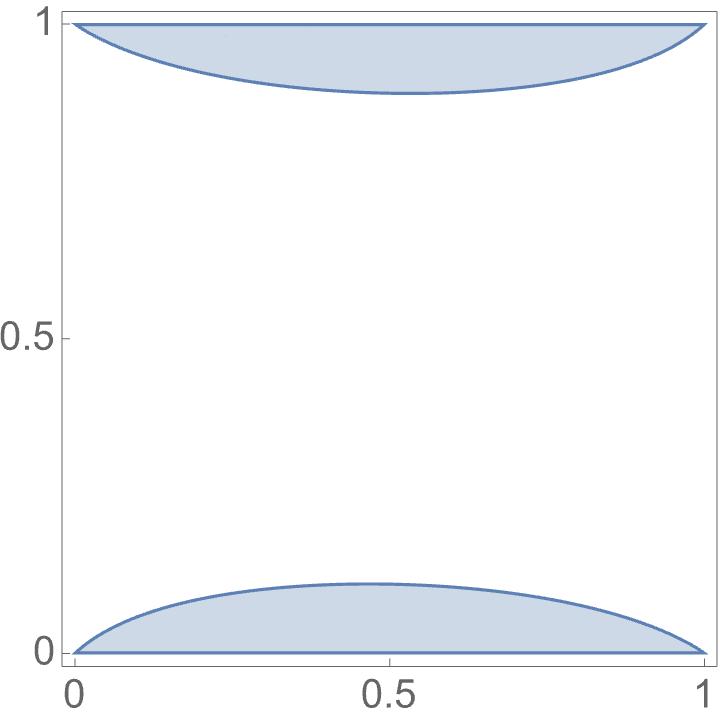}\hspace{.8em}
\caption[]{}
\label{fig:region2}
\end{subfigure}
\begin{subfigure}{.29\textwidth}
\centering
\includegraphics[scale=0.25]{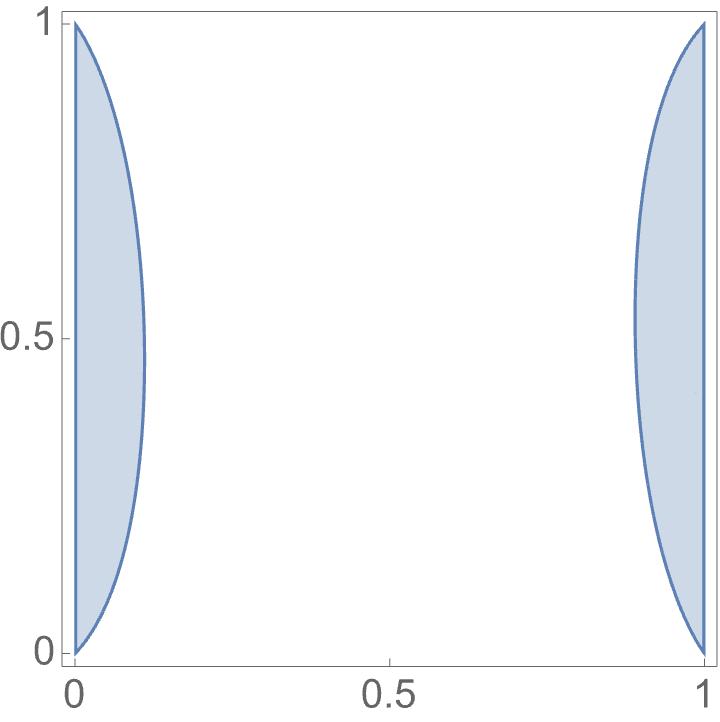}\hspace{.8em}
\caption[]{}
\label{fig:region3}
\end{subfigure}
\captionsetup{singlelinecheck=off}
\vspace{-0.5em}\caption{Theorem~\ref{thm:maincommon} establishes two-point concentration around~$n_N$ 
of the maximum common induced subgraph of two independent binomial random graphs~$G_{N,p_1}$ and~$G_{N,p_2}$. 
The plots illustrate how the form of~$n_N$ subdivides the unit square of edge-probabilities ${(p_1,p_2)\in (0,1)^2}$: 
we have ${\nN= \xot_N(\ps)\sim 4\log_{b_0(\hat{p})} N}$ in~region~(a), 
${\nN\sim \xone_N(p^*_1) \sim 2\log_{b_1(p^*_1)}N}$ in~region~(b), 
and ${\nN\sim \xtwo_N(p^*_2) \sim 2\log_{b_2(p^*_2)}N}$ in~region~(c), 
where~$\ps$ is defined as in~\eqref{def:phat} and~$p^*_i$ is the unique solution of~$\log b_0(p)=2\log b_i(p)$; 
see~Lemma~\ref{lem:figcaption} in~Appendix~\ref{sec:calcstuff}. 
Previous work~\cite{chatterjee2021isomorphisms} applied to the special case~$p_1=p_2=1/2$, i.e., only to the central point in region~(a).%
\label{fig:regions}}
\end{figure}

As a consequence of our proof of Theorem~\ref{thm:contain}, 
we are able to determine the asymptotic distribution of the number of induced copies of $G_{n,p_1}$ in $G_{N,p_2}$, 
resolving another open problem of Chatterjee and Diaconis~\cite{chatterjee2021isomorphisms}. 
In concrete words, Theorem~\ref{thm:contain:dist}~\ref{enum:contain:dist:12} shows that 
the number of induced copies has a Poisson distribution 
for~${p_2=1/2}$ and~$n$ close to the sharp threshold location~${2\log_a N + 1}$. 
Furthermore, Theorem~\ref{thm:contain:dist}~\ref{enum:contain:dist:other} shows that  
the number of induced copies has a `squashed' log-normal distribution
for~${p_2 \neq 1/2}$ and ${n=2\log_a N + \Theta(1)}$, 
which is a rather unusual limiting distribution for random discrete structures 
(that intuitively arises since the number of such induced copies is so strongly dependent on the number of edges in~$G_{n,p_1}$; see Section~\ref{sec:contain:dist}). 
Below~${(N)_n := N(N-1)\dots (N-n+1)}$ denotes the falling factorial. 
\enlargethispage{\baselineskip}
\begin{thm}[Asymptotic distribution]\label{thm:contain:dist}%
Let~$p_1,p_2\in (0,1)$ be constants.  
Define $a := {1/\bigpar{p_2^{p_1}(1-p_2)^{1-p_1}}}$ and ${\varepsilon_N := (\log \log N)^2/\log N}$ as in~Theorem~\ref{thm:contain}. 
For independent binomial random graphs~$G_{n,p_1}$ and~$G_{N,p_2}$, let~$X=X_{n,N}$ denote the number of induced copies of~$G_{n,p_1}$ in~$G_{N,p_2}$. 
Then the following~holds, as~$N \to \infty$:%
\begin{romenumerate}
\item\label{enum:contain:dist:12}
If~$p_2=1/2$ and~$2\log_a N -1 + \eps_N \le n \le N $, then~$X$ has asymptotically Poisson distribution with mean~$\mu=\mu_{n,N}:=(N)_{n}2^{-\binom{n}{2}}  = \bigpar{N \cdot 2^{-(n-1)/2+O(n/N)} }^{n}$, i.e., 
\begin{align}\label{eq:contain:dist:12}
\TV\bigpar{X, \: \Po\xpar{\mu}} \: &\to \: 0.
\intertext{%
\item 
If~$p_2 \neq 1/2$ and~$n-(2\log_a N+1) \to c \in (-\infty,\infty)$ as~$N \to \infty$, then}
\label{eq:contain:dist:other}
\frac{\log (1+X)}{\log N}\: &\dto \: \TNor\Bigpar{-c, \; 2p_1(1-p_1)  {\log^2_a(1/p_2-1)}},
\end{align}
where~$\TNor(\mu,\sigma^2)$ has cumulative distribution function~$F(x) := {\indic{x \ge 0} \Pr(\Nor(\mu,\sigma^2) \le x)}$.
\end{romenumerate}\vspace{-0.125em}%
\end{thm}
%

For the uniform case $p_2=1/2$, Svante Janson pointed out that one can sharpen the containment probability estimate from Theorem~\ref{thm:contain}~\ref{enum:contain:12} a little further, 
by simply invoking the limiting Poisson result~\eqref{eq:contain:dist:12} for those~$n$ to which~\eqref{eq:contain:12} does not apply.  
This readily leads to the following refinement of Theorem~\ref{thm:contain}~\ref{enum:contain:12}, 
which (a)~shows that in~\eqref{eq:contain:12} we only need to assume ${\eps_N \cdot \log N \to \infty}$,
and (b)~also determines the limiting probability inside the `transition window' 
that was missing in~\eqref{eq:contain:12} and previous work of Chatterjee and Diaconis~\cite{chatterjee2021isomorphisms}. 
\begin{cor}[Refined sharp threshold in special case~$p_2=1/2$]\label{cor:contain}%
Let~$p_1\in (0,1)$ be a constant. 
Then the following refinement of~\eqref{eq:contain:12} holds, for independent binomial random graphs~$G_{n,p_1}$ and~$G_{N,1/2}$:%
\begin{equation}
\lim_{N \to \infty}\Pr\bigpar{G_{n,p_1}\isub G_{N,1/2}} = \begin{cases} 1 \ \ \ \ & \text{if $\bigpar{n-(2\log_2 N+1)}\log_2 N \to -\infty$,}\\ 
1-e^{-2^{-c}} & \text{if $\bigpar{n-(2\log_2 N+1)}\log_2 N \to c \in (-\infty,\infty)$,}\\
0 \ \ \ \ & \text{if $\bigpar{n-(2\log_2 N+1)}\log_2 N \to \infty$.}\end{cases}%
\end{equation}%
\end{cor}

%

%
\pagebreak[3]
\subsection{Maximum common induced subgraph problem for random graphs}\label{sec:introcommon} 
In the \emph{maximum common induced subgraph problem} the objective is to determine the maximum size of an induced subgraph of~$F$ that is isomorphic to an induced subgraph of~$G$, where the maximum size is with respect to the number of vertices (this generalizes the induced subgraph isomorphism problem, since the maximum common induced subgraph is~$F$ if and only if~$F$ is isomorphic to an induced subgraph of~$G$). 
In this paper we focus on this problem for two independent binomial random graphs~$F=G_{N,p_1}$ and~$G=G_{N,p_2}$, 
with constant edge-probabilities~$p_1,p_2 \in (0,1)$.
One motivation here comes 
from combinatorial probability theory~\cite{chatterjee2021isomorphisms}, where the following paradox was recently pointed out: 
two independent infinite random graphs with edge-probabilities~$p_1,p_2 \in (0,1)$ are isomorphic with probability one, 
but two independent finite random graphs~$G_{N,p_1}$ and~$G_{N,p_2}$ are isomorphic with probability tending to zero as~$N \to \infty$.
This discontinuity of limits raises the question of finding the size of the maximum common induced subgraph of~$G_{N,p_1}$ and~$G_{N,p_2}$, 
which also is a natural random graph problem in its own right. 
Chatterjee and Diaconis~\cite{chatterjee2021isomorphisms} answered this question in the special case~$p_1=p_2=1/2$. 
In this paper we resolve the general case with edge-probabilities~${(p_1,p_2) \in (0,1)^2}$: 
we discover that the general form of the maximum size is significantly more complicated than for the uniform case, 
which in fact is closely linked to large variance~difficulties. 

Our next result establishes two-point concentration of the size~$I_N$ of the maximum common induced subgraph of two independent binomial random graphs~$G_{N,p_1}$ and~$G_{N,p_2}$, 
resolving an open problem of Chatterjee and Diaconis~\cite{chatterjee2021isomorphisms}. 
In concrete words, Theorem~\ref{thm:maincommon} shows that~$I_N$ equals, with probability tending to one as~$N \to \infty$, one of (at most) two consecutive integers;  
see~\eqref{eq:thmmain}--\eqref{def:nN:asymp} below.  
%
%
\begin{thm}[Two-point concentration]\label{thm:maincommon}%
Let~$p_1,p_2\in (0,1)$ be constants.  
For independent binomial random graphs~$G_{N,p_1}$ and~$G_{N,p_2}$, define~$I_{N}$ as the size of the maximum common induced subgraph. 
Then 
\begin{equation}\label{eq:thmmain}
\lim_{N \to \infty}\Pr\bigpar{I_N \in \{\floor{n_N-\eps_N}, \floor{n_N+\eps_N}\}} = 1 ,
\end{equation}
where $\eps_N:=(\log \log N)^2/\log N$
and the parameter~$n_N=n_N(p_1,p_2)$  defined in Remark~\ref{rem:savingold} satisfies
\begin{equation}\label{def:nN:asymp}
n_N 
=
\frac{4\log N+O(\log \log N)}{\min_{p\in[0,1]} \max\bigcpar{\log b_0(p), \: 2\log b_1(p), \: 2\log b_2(p)}} ,
\end{equation}
where, using the convention~$0^0=1$, we have
\begin{equation}
\label{def:b}
{b_0} = {b_0}(p):= \lrpar{\frac{p}{p_1p_2}}^p \lrpar{\frac{1-p}{ (1-p_1)(1-p_2)}}^{1-p}
\quad \text{ and } \quad 
b_i = b_i(p) := \lrpar{\frac{p}{p_i}}^p \lrpar{\frac{1-p}{ 1-p_i}}^{1-p}.
\end{equation}
\end{thm}
%
%
Interestingly, Figure~\ref{fig:regions} shows that the form of the two-point concentration location~$n_N=n_N(p_1,p_2)$ changes as we vary the edge-probabilities~$(p_1,p_2) \in (0,1)^2$, 
which is a rather surprising  
phenomenon for random graphs with constant edge-probabilities. 
In Section~\ref{sec:heuristic:MCIS} we discuss how the three different forms of~$n_N$ heuristically arise (due to containment in~$G_{N,p_1}$, containment in~$G_{N,p_2}$, and containment in both).  
The value of~$p$ which attains the minimum in~\eqref{def:nN:asymp} 
is an approximation of the edge-density of a maximum common induced subgraph in~$G_{N,p_1}$ and~$G_{N,p_2}$. 
There is a natural guess for the `correct' edge-density: if we condition on 
$G_{n,p_1}$ and $G_{n,p_2}$ being equal, then using linearity of expectation 
the expected edge-density~equals
\begin{equation}\label{def:phat}
\E\biggpar{\frac{e\bigpar{G_{n,p_1}}}{\binom{n}{2}} \: \bigg| \: G_{n,p_1} =G_{n,p_2}}=\frac{p_1p_2}{p_1p_2+(1-p_1)(1-p_2)} = : \ps(p_1,p_2) = \ps.
\end{equation}
It turns out that $\ps$ is indeed the correct edge-density
for a range of edge-probabilities~${(p_1,p_2) \in (0,1)^2}$; see region~(a) in Figure~\ref{fig:regions}. 
In those cases Corollary~\ref{cor:equalprob} gives the two-point concentration location explicitly (the derivation is deferred to Appendix~\ref{sec:calcstuff}), 
which in the special case~$p_1=p_2=1/2$ 
recovers~\cite[Theorem~1.1]{chatterjee2021isomorphisms}. 
\begin{cor}[Special cases]\label{cor:equalprob}%
Let~$p_1,p_2\in (0,1)$ be constants 
satisfying~$\log b_0(\hat{p}) > \max\bigcpar{2\log b_1(\hat{p}), \: 2\log b_2(\hat{p})}$ for~$b_j(p)$ is as defined in~\eqref{def:b}, which in particular holds when~$p_1=p_2$.  
For independent binomial random graphs~$G_{N,p_1}$ and~$G_{N,p_2}$, 
define~$I_{N}$ as the size of the maximum common induced subgraph. 
Then 
\begin{equation}\label{eq:equalprob}
\lim_{N \to \infty}\Pr\bigpar{I_N \in \bigcpar{\bigfloor{\xot_N(\ps)-\eps_N}, \bigfloor{\xot_N(\ps)+\eps_N}}} = 1 ,
\end{equation}
where~$\eps_N:=(\log \log N)^2/\log N$, $\ps=\ps(p_1,p_2)$ is defined as in~\eqref{def:phat}, $b_0=b_0(p)$ is defined as in~\eqref{def:b}, and 
\begin{align}
\label{def:x_N}
\xot_N(p) &:=  4\log_{b_0} N - 2\log_{b_0}\log_{b_0} N-2\log_{b_0} (4/e) +1.
\end{align}
\end{cor}
In general, the two-point concentration location~$n_N=n_N(p_1,p_2)$ is defined in~\eqref{def:nN} 
as the solution of an optimization problem over all edge-densities~$p \in [0,1]$ of a potential maximum common subgraph 
(as~indicated above, in~Corollary~\ref{cor:equalprob} the maximum is attained by $\hat{p}$, i.e., we have~$n_N=\xot_N(\ps)$ in~\eqref{def:nN} below; see Lemma~\ref{lem:figcaption} in Appendix~\ref{sec:calcstuff}). 
In Section~\ref{sec:heuristic:MCIS} we discuss how this more complicated form of~$n_N$ stems from 
large variance difficulties that can arise in the second moment method arguments 
(where the typical number of copies can be zero even when the expected number of copies tends to infinity). 
While the definition~\eqref{def:nN} of~$n_N$ involves two implicitly defined~\eqref{def:xi} parameters~$x^{(i)}_N(p)$, 
we remark that when~$n_N=x^{(i)}_N(p)$ holds for $i\in \{1,2\}$, 
then~$x^{(i)}_N(p)$ has the explicit form\footnote{
One might be tempted to think that the estimate~\eqref{def:y_N:asymp} could be used 
to explicitly define~$x^{(i)}_N(p)$ for~$i \in \{1,2\}$, by simply ignoring the additive error term. 
Unfortunately, this does not work for a subtle technical reason: 
for fixed~$N$ this would lead to ${\lim_{p \to p_i}x^{(i)}_{N}(p)=-\infty}$ (since~${\lim_{p \to p_i}b_i(p)=b_i(p_i)=1}$), 
making it inadequate for the definition~\eqref{def:nN} of~$n_N$.} 
given in~\eqref{def:y_N:asymp} below; 
see Lemma~\ref{lem:figcaption} in Appendix~\ref{sec:calcstuff} for further estimates~of~$n_N$.
\begin{remark}[Explicit expression for~$n_N$]\label{rem:savingold}%
The proof of Theorem~\ref{thm:maincommon} uses 
\begin{equation}\label{def:nN}
n_N=n_N(p_1,p_2) :=\max_{p\in [0,1]} \; \min\left\{\xot_N(p), \: \xone_N(p), \: \xtwo_N(p) \right\} ,
\end{equation}
where~$\xot_N(p)$ is as defined in~\eqref{def:x_N}, 
and~$x^{(i)}_N=x^{(i)}_N(p)$ with~$i\in\{1,2\}$ is the unique solution to 
\begin{equation}\label{def:xi}
    e\cdot N \; = \; x^{(i)}_N(p) \cdot b_i^{\bigpar{x^{(i)}_N(p)-1}/2},
\end{equation}
where~$b_i=b_i(p)$ is as defined in~\eqref{def:b}. 
Furthermore, if  $n_N=x^{(i)}_N(p)$ holds in~\eqref{def:nN} for $i\in \{1,2\}$, then 
\begin{align}
\label{def:y_N:asymp}
x^{(i)}_N(p) & = 2\log_{b_i}N-2\log_{b_i}\log_{b_i}N-2\log_{b_i}(2/e) +1 +O\lrpar{ \frac{\log\log N}{\log N}}.
\end{align}
\end{remark}

\subsection{Intuition and proof heuristics}\label{sec:heur} 
The proofs of our main results are based on involved applications of the first and second moment method, 
which each require an extra twist due to large variance (that prevents standard applications of these methods). 
In the following we highlight some of the key intuition and heuristics going into our arguments.

\subsubsection{Induced subgraph isomorphism problem}\label{sec:heur:ISIP}
We now heuristically discuss the sharp threshold for induced containment~${G_{n,p_1} \isub G_{N,p_2}}$. 
Besides outlining the reasoning behind our proof approach for \refT{thm:contain}, we also motivate why the threshold is located around~${n \approx 2\log_a N + 1}$, 
and clarify why the case~${p=1/2}$ behaves so differently than the case~${p \neq 1/2}$. 

The natural proof approach would be to apply the first and second moment method to the random variable~$X$ that counts the number of induced copies of~$G_{n,p_1}$ in~$G_{N,p_2}$. 
While this approach can indeed be used to establish~\eqref{eq:contain:12} when~${p_2=1/2}$ (as done in~\cite{chatterjee2021isomorphisms} for~$p_1=p_2=1/2$), 
it fails when~${p_2 \neq 1/2}$: the reason is~$\Var X \gg (\E X)^2$, i.e., that the variance of~$X$ is too large to apply the second moment method. 

We overcome this second moment challenge for~$p_2 \neq 1/2$ by identifying the key reason for the large variance, 
which turns out to be random fluctuations of the number of edges in~$G_{n,p_1}$.
To work around the effect of these edge-fluctuations 
we use a multi-round exposure approach, where we first reveal~$G_{n,p_1}$ and then~$G_{N,p_2}$ 
(which conveniently allows us to deal with one source of randomness at a time). 
When we reveal~$G_{n,p_1}$, we exploit that~$H:=G_{n,p_1}$ will typically be asymmetric and satisfy other pseudorandom properties. 
Writing~$X_H$ for the number of induced copies of~$H$ in~$G_{N,p_2}$, 
we then focus on the conditional~probability
\begin{equation}\label{eq:heur:ISI:Pr}
\pr(G_{n,p_1} \isub G_{N,p_2} \mid G_{n,p_1}=H) = \Pr(H\isub G_{N,p_2})=\pr(X_H \ge 1) .
\end{equation}
To see how the number of edges in~$H$ affects this containment probability, 
suppose for concreteness that~$H$ has $m=\binom{n}{2}p_1+\delta n/2$ edges. 
Recalling~$a = {1/\bigpar{p_2^{p_1}(1-p_2)^{1-p_1}}}$, 
it turns out (see~\eqref{eq:firstmomentXH}--\eqref{eq:mu} in Section~\ref{sec:proofcontain} for the routine details) that the expected number of induced copies of~$H$ in~$G_{N,p_2}$ satisfies 
\begin{equation}\label{eq:heur:ISI}
\E X_H= 
(N)_n \cdot  p_2^m(1-p_2)^{\binom{n}{2}-m}
= \biggsqpar{\lrpar{\frac{p_2}{1-p_2}}^{\delta/2} N e^{-\log a \cdot (n-1)/2 + O(n/N)} }^n.
\end{equation}
Intuitively, the expected number~${\E X_H}$ depends on~$H$ only through its number of vertices~${n=v(H)}$ and edges~${m=e(H)}$, 
since (a)~for asymmetric graphs~$H$ the number~${(N)_n}$ of potential copies depends only on~$n$, 
and (b)~the containment probability~${p_2^m(1-p_2)^{\binom{n}{2}-m}}$ of any such copy of~$H$ depends only on~$n$ and~$m$. 

For~$n=2\log_a N+1+c$ and $p_2 \neq 1/2$ it follows from~\eqref{eq:heur:ISI} that 
the value of the edge-deviation parameter~$\delta$ determines whether~$\E X_H$ 
goes to~$\infty$ or~$0$, 
i.e., depending on whether~${\delta \log\bigpar{p_2/(1-p_2)}}$ is smaller or larger than~${c \log a}$ 
(here we also see why the case $p_2=1/2$ behaves so differently: in~\eqref{eq:heur:ISI} the term involving~$\delta$ equals one and thus disappears). 
With some technical effort, we can make the first moment 
heuristic that~$\E X_H\to\infty$ implies~$\pr(X_H \ge 1) \to 1$ rigorous, 
i.e., use the first and second moment method to show~that 
\begin{equation}\label{eq:heur:ISI:cases}
\pr(X_H \ge 1) = \begin{cases} o(1) \quad & \text{if~$\delta \log\bigpar{p_2/(1-p_2)}< c \log a$,}\\ 
1-o(1) \quad & \text{if~$\delta \log\bigpar{p_2/(1-p_2)}>c \log a$,}\end{cases} 
\end{equation}
where a good control of the `overlaps' of different~$H$ in the variance calculations requires us to identify and exploit suitable pseudorandom properties of~$H$ 
(which we can `with foresight' insert into our argument when we first reveal~${H=G_{n,p_1}}$). 
The crux is that the containment conditions in~\eqref{eq:heur:ISI:cases} only depend on deviations in the number~$e(G_{n,p_1})={\binom{n}{2}p_1+\delta n/2}$ of edges in~$G_{n,p_1}$, 
which in view of~\eqref{eq:heur:ISI:Pr} intuitively translates~into 
\begin{equation}\label{eq:heur:ISI:Pr2}
\pr(G_{n,p_1} \isub G_{N,p_2}) =
  \Pr\biggpar{\frac{e(G_{n,p_1})-\tbinom{n}{2}p_1}{n/2} \cdot \log\bigpar{p_2/(1-p_2)} > c \log a} + o(1)
\end{equation}
for~$n=2\log_a N+1+c$. 
This in turn makes the threshold result~\eqref{eq:contain:main} plausible 
via the Central Limit Theorem 
(in fact, \eqref{eq:heur:ISI:Pr2} is also consistent with the form of~\eqref{eq:contain:12} for~$p_2=1/2$, since then~${\log\bigpar{p_2/(1-p_2)}=0}$); 
see Section~\ref{sec:proofcontain} for the full technical~details of the proof of \refT{thm:contain}.

With this knowledge about~$X_H$ in hand, the distributional result Theorem~\ref{thm:contain:dist} follows without much extra conceptual work. 
Indeed, for~$p \neq 1/2$ we obtain the unusual limiting distribution~\eqref{eq:contain:dist:other} 
by exploiting that the event~$X_H \ge 1$ in the second moment method based~$1-o(1)$ statement in~\eqref{eq:heur:ISI:cases} can be strengthened to~$X_H \approx \E X_H$; see Section~\ref{sec:contain:dist}. 
Furthermore, for~$p=1/2$ we obtain the Poisson limiting distribution~\eqref{eq:contain:dist:12} by refining our variance estimates in order to apply the Stein-Chen method to~$X_H$; see Section~\ref{sec:contain:dist:poisson}.

\subsubsection{Maximum common induced subgraph problem}\label{sec:heuristic:MCIS}
We now heuristically discuss the two-point concentration of the size~$I_N$ of the maximum common induced subgraph in~$G_{N,p_1}$ and~$G_{N,p_2}$. 
Our main focus is on motivating the `correct' typical value~$I_N \approx n_N$ in Theorem~\ref{thm:maincommon},   
which requires a refinement of the vanilla first and second moment approach. 
For ease of exposition, here we shall restrict our attention to the first-order asymptotics of~$n_N=n_N(p_1,p_2)$ from~\eqref{def:nN:asymp}.

Armed with the first and second moment insights for~$X_H$ from the heuristic discussion in~\refS{sec:heur:ISIP},
the natural approach towards determining the typical value of~$I_N$ would be to focus 
on the random variables~$X_{n,m} := \sum_{H \in \cG_{n,m}}X^{(1)}_HX^{(2)}_H$, 
where~$\cG_{n,m}$ denotes the set of all unlabeled graphs with~$n$ vertices and~$m$ edges, 
and~$X^{(i)}_H$ counts the number of induced copies of~$H$ in~$G_{N,p_i}$.  
Here the crux is that 
\begin{equation}\label{eq:heu:MCIS:IN}
I_N = \max\bigcpar{ n \: : \: \text{$X_{n,m} \ge 1$ for some~$0 \le m \le \tbinom{n}{2}$}} . 
\end{equation}
Using~$|\cG_{n,m}|=\binom{\binom{n}{2}}{m} e^{O(n\log n)}$ and the independence of~$G_{N,p_1}$ and~$G_{N,p_2}$, 
similar to~\eqref{eq:heur:ISI} it turns out (see~\eqref{eq:keyestimate:XH1XH2:1} and~\eqref{eq:mu:EX:asymp} in Section~\ref{sec:commonfirst} for the details) 
that for edge-density~$p=p(m,n):=m/\binom{n}{2}$ we have 
\begin{equation}\label{eq:heu:MCIS:X0nm}
\begin{split}
\E X_{n,m} 
&= 
e^{O(n\log n)} \binom{\binom{n}{2}}{m} \cdot \prod_{i \in \{1,2\}} (N)_n p_i^m(1-p_i)^{\binom{n}{2}-m}
    =\left[N^2e^{-\frac{n}{2} \log {b_0}(p)+O(\log n)}\right]^n.
\end{split}
\end{equation}
Taking into account when these expectations go to~$0$ and~$\infty$, 
in view of~\eqref{eq:heu:MCIS:IN} and standard first moment heuristics it then is natural to guess\footnote{The first moment heuristic for~$n_N$ is as follows.  
For any~$n \ge (1+\xi)n_N$, the definition of~$n_N$ and~\eqref{eq:heu:MCIS:X0nm} imply that~$\sum_m\E X_{n,m} \le n^2 \cdot N^{-\Theta(\xi n)} \to 0$. 
Similarly, for any~$n \le (1-\xi)n_n$ we have~$\E X_{n,m} \ge N^{\Theta(\xi n)} \to \infty$ for~$m=\binom{n}{2}\ps$.} 
that the typical value of~$I_N$ should be approximately
\begin{equation}\label{eq:heu:MCIS:nN:0}
n_N \approx \max_{p \in [0,1]}4\log_{b_0(p)} N = \frac{4 \log N}{\min_{p \in [0,1]} \log b_0(p)} = 4 \log_{b_0(\ps)} N ,
\end{equation}
where~$\ps=\ps(p_1,p_2)$ from~\eqref{def:phat} turns out to be the unique minimizer of~$\log b_0(p)$ from~\eqref{def:b}. 
For this choice of~$n_N$ it turns out that we indeed 
typically have~${I_N \approx n_N}$ for a range of edge-probabilities ${(p_1,p_2) \in (0,1)^2}$, including the special case~$p_1=p_2$; 
see Corollary~\ref{cor:equalprob} and region~(a) in~Figure~\ref{fig:regions}.

For general edge-probabilities~${(p_1,p_2) \in (0,1)^2}$ the `correct' form of~$n_N$ is more complicated than our first guess~\eqref{eq:heu:MCIS:nN:0},
and the key reason turns out to be that due to large variance we can have~${\Pr(X_{n,m} \ge 1) \to 0}$ despite~${\E X_{n,m} \to \infty}$.
We overcome this difficulty by realizing that containment in~$G_{N,p_i}$ can, from a `first moment perspective', sometimes be harder than containment in both~$G_{N,p_1}$ and~$G_{N,p_2}$, 
i.e., that we can have~${\E X^{(i)}_{n,m} \to 0}$ for~$X^{(i)}_{n,m} := \sum_{H \in \cG_{n,m}}X^{(i)}_H$ despite~${\E X_{n,m} \to \infty}$.  
Similarly to~\eqref{eq:heu:MCIS:X0nm} and~\eqref{eq:heur:ISI} it turns out (see~\eqref{eq:keyestimate:XHi:1} in Section~\ref{sec:commonfirst} for the details)
that for edge-density~$p=p(m,n)=m/\binom{n}{2}$ we~have 
\begin{equation}\label{eq:heu:MCIS:Xinm}
\begin{split}
\E X^{(i)}_{n,m} 
&= 
e^{O(n\log n)} \binom{\binom{n}{2}}{m} \cdot (N)_n p_i^m(1-p_i)^{\binom{n}{2}-m}
    =\left[Ne^{-\frac{n}{2} \log {b_i}(p)+O(\log n)}\right]^n.
\end{split}
\end{equation}
Note that~$X_{n,m} \ge 1$ implies both~$X^{(1)}_{n,m} \ge 1$ and $X^{(2)}_{n,m} \ge 1$. 
Taking into account when the two expectations in~\eqref{eq:heu:MCIS:Xinm} go to~$0$, 
we thus refine our first guess~\eqref{eq:heu:MCIS:nN:0} for the typical value of~$I_N$ to approximately 
\begin{equation}\label{eq:heu:MCIS:nN}
\begin{split}
n_N & \approx \max_{p \in [0,1]}\min\bigcpar{4\log_{b_0(p)} N, \: 2\log_{b_1(p)} N, \: 2\log_{b_2(p)} N} \\
& = \frac{4 \log N}{\min_{p \in [0,1]} \max\bigcpar{\log b_0(p), \: 2\log b_1(p), \: 2\log b_2(p)}} .
\end{split}
\end{equation}
This turns out to be the `correct' value of~$n_N$ up to second order correction terms, see~\eqref{def:nN:asymp} and~Figure~\ref{fig:regions}. 
Indeed, with substantial technical effort (more involved than for the induced subgraph isomorphism problem) 
we can use refinements of the first and second moment method to prove two-point concentration of form~${I_N \approx n_N}$ 
for all edge-probabilities~${(p_1,p_2) \in (0,1)^2}$; 
see \refS{sec:common} for the full technical details of the proof of \refT{thm:maincommon}.

\subsubsection{Pseudorandom properties of random graphs}\label{sec:pseudorandom}
%
To enable the technical variance estimates in the proofs of our main results, we need to get good control over the `overlaps' of different induced copies. 
To this end it will be important to take certain pseudorandom properties of random graphs into account, 
by restricting our attention to (labelled) graphs~$H$ from the following three~classes:%
{\begin{itemize}[leftmargin=4em]
\itemsep 0.125em \parskip 0em  \partopsep=0pt \parsep 0em 
\item[$\cA_n$:]
The set of all $n$-vertex graphs~$H$ where, for all vertex-subsets~${L\subseteq [n]}$ of size~${|L| \ge n-n^{2/3}}$, the induced subgraph~$H[L]$ is asymmetric, i.e., 
its automorphism group satisfies~${|\Aut(H[L])|=1}$. 
\item[$\cE_{n,m}$:]
The set of all $n$-vertex graphs~$H$ where, for all non-empty vertex-subsets~${L\subseteq [n]}$,  the induced  subgraph~$H[L]$ contains $e(H[L])=\tbinom{|L|}{2}m/\tbinom{n}{2} \pm n^{2/3}(n-|L|)$ edges.
\item[$\cF_{n,p}$:]
The set of all $n$-vertex graphs~$H$ with~$e(H)=\tbinom{n}{2}p \pm n^{4/3}$ edges. 
\end{itemize}}%
In concrete words, the first property~$\cA_n$ formalizes the folklore fact that all very large subsets of random graphs are asymmetric; see~\cite{BB,ER1963,kim2002asymmetry}.
The second and third properties~$\cE_{n,m}$ and~$\cF_{n,p}$ formalize the well-known fact that the edges of a random graph are well-distributed; see~\cite{balister2019dense,BB,FKM2014}.
The following auxiliary lemma states that these are indeed typical properties of dense random graphs 
(we defer the proof to~\refA{apx:pseudorandom}, since it is rather tangential to our main arguments). 
Below~$G_{n,m}$ denotes the uniform random graph, which is chosen uniformly at random from all~$n$-vertex graphs with exactly~$m$ edges. 
Furthermore, the abbreviation whp (with high probability) means with probability tending to one as the number of vertices goes to~infinity. 
\begin{lemma}[Random graphs are pseudorandom]\label{lem:typical}
Let~$p,\gamma \in (0,1)$ be constants. Then the following~holds:
\begin{romenumerate}
\item For all~$\gamma \tbinom{n}{2} \le m \le (1-\gamma)\tbinom{n}{2}$, the uniform random graph whp satisfies~$G_{n,m} \in \cA_n \cap \cE_{n,m}$.
\item The binomial random graph whp satisfies~$G_{n,p} \in \cA_n \cap \cF_{n,p}$.
\end{romenumerate}\vspace{-0.125em}%
\end{lemma}
The reason for why we only consider the edge-property~$\cE_{n,m}$ for the uniform random graph~$G_{n,m}$ is conceptual: the edges of~$G_{n,m}$ are simply more well-distributed than the edges of the binomial random~$G_{n,p}$ with~$p:=m/\binom{n}{2} \in [\gamma,1-\gamma]$, where for large vertex-subsets~$L \subseteq [n]$ of size~$|L|= n - o(n^{1/3})$ we expect that typically $\bigl|e(G_{n,p}[L])-\tbinom{|L}{2}p\bigr| = \Omega\bigpar{|L|\sqrt{p(1-p)}} \gg n^{2/3}(n-|L|)$ holds.


\pagebreak[3]
\section{Induced subgraph isomorphism problem}\label{sec:contain}
This section is devoted to the induced subgraph isomorphism problem for two independent random graphs~$G_{n, p_1}$ and~$G_{N,p_2}$ with constant edge-probabilities~${p_1,p_2 \in (0,1)}$. 
It naturally splits into several parts: in the core \refS{sec:proofcontain} we prove the sharp threshold result \refT{thm:contain}, and in the subsequent Sections~\ref{sec:contain:dist}--\ref{sec:contain:dist:poisson} we then separately prove the two parts of the distributional result \refT{thm:contain:dist}. 
%

\subsection{Sharp Threshold: Proof of \refT{thm:contain}}\label{sec:proofcontain}
In this section we prove \refT{thm:contain}, i.e., establish a sharp threshold for the appearance of an induced copy of~$G_{n,p_1}$ in~$G_{N,p_2}$, where~$p_1,p_2 \in (0,1)$. 
Our proof-strategy is somewhat roundabout,
since we need to deal with the difficulty that for~$p_2 \neq 1/2$ the induced containment event $G_{n,p_1}\isub G_{N,p_2}$ is sensitive towards small variations in the number of edges of~$G_{n,p_1}$ (as discussed in Sections~\ref{sec:intro:contain} and~\ref{sec:heur:ISIP}).
For this reason we first prove a sharp threshold result for the event that~$G_{N,p_2}$ contains an induced copy of a particular $n$-vertex pseudorandom graph~$H$ with~$m$ edges (see \refL{lem:maincontain} below). 
Since~$G_{n,p_1}$ will typically be pseudorandom by \refL{lem:typical}, 
this then effectively reduces the problem of estimating ${\Pr(G_{n,p_1}\isub G_{N,p_2})}$ 
to understanding deviations in the number of edges of~$G_{n,p_1}$, 
which is a well-understood and much simpler~problem. 

Turning to the details, to restrict our attention to graphs with pseudorandom properties 
we~introduce
\begin{equation}\label{def:Tnm}
\cT_{n,m} := \cA_n \cap \cE_{n,m},
\end{equation}
where the asymmetry property~$\cA_n$ and the edge-distribution property~$\cE_{n,m}$ are defined as in \refS{sec:pseudorandom}. 
For edge-counts~$m$ of interest, 
\refL{lem:typical} shows that the uniform random graph whp satisfies~$G_{n,m} \in \cT_{n,m}$, 
so \refL{lem:maincontain}  effectively gives a sharp threshold result for the induced containment~$G_{n,m} \isub G_{N,p_2}$. 
\begin{lemma}[Sharp threshold for pseudorandom graphs]\label{lem:maincontain}
Let $p_1,p_2\in (0,1)$ be constants.
Define $a := {1/(p_2^{p_1}(1-p_2)^{1-p_1})}$, $\eps_N:={(\log \log N)^2/\log N}$, and~$\psi:={\log_a\bigpar{p_2/(1-p_2)}}$.
If~$\delta_m := {\bigsqpar{m-\tbinom{n}{2}p_1}/(n/2)}$ 
and~$c_N:={n-(2\log_a N + 1)}$
satisfy ${\abs{\psi \delta_m} = o(n)}$ and~${\abs{c_N} = o(\log_a N)}$, then the following holds: as~$N \to \infty$, for any graph~$H \in \cT_{n,m}$:
 \begin{equation}\label{eq:p2}
\Pr\bigpar{H\isub G_{N,p_2}} = \begin{cases} o(1) \quad & \text{if~$\psi \delta_m-c_N \le -\eps_N$,}\\ 
1-o(1) \quad & \text{if~$\psi \delta_m-c_N \ge \eps_N$.}\end{cases}
\end{equation}
\end{lemma}
\begin{remark}\label{rem:0-statement}
The~\mbox{$o(1)$ statement} in~\eqref{eq:p2} remains valid for any $n$-vertex graph~$H$ with $e(H)=m$ edges. 
\end{remark}
Before giving the (first and second moment based) proof of \refL{lem:maincontain}, we first show how it implies \refT{thm:contain} using a multi-round exposure argument, where we first reveal the number of edges of~$G_{n,p_1}$, 
then reveal the random graph~$G_{n,p_1}$, and afterwards try to embed~$G_{n,p_1}$ into the random graph~$G_{N,p_2}$, each time discarding atypical outcomes along the way (so that we can focus on pseudorandom~$G_{n,p_1}$). 
\begin{proof}[Proof of~\refT{thm:contain}]
To estimate the probability that~$G_{n,p_1} \isub G_{N,p_2}$ holds, 
by monotonicity\footnote{A standard coupling argument shows that~${\Pr(G_{n,p_1} \isub G_{N,p_2})}$ is monotone decreasing in~$n$. 
The crux is that ${G_{n+1,p_1}[\{1, \ldots, n\}]}$, i.e., the induced subgraph of $G_{n+1,p_1}$ on vertex~set ${\{1,\ldots, n\}}$, has the same distribution as~$G_{n,p_1}$.
It thus follows that $\Pr(G_{n+1,p_1} \isub G_{N,p_2}) \le \Pr({G_{n+1,p_1}[\{1, \ldots, n\}]} \isub G_{N,p_2}) = \Pr(G_{n,p_1} \isub G_{N,p_2})$, as claimed.} 
in~$n$ it suffices to consider the case where~$c_N = o(\log \log N)$ holds 
(in fact~$c_N=O(1)$ suffices, since~$\lim_{c \to \infty}f(c)=0$ and~$\lim_{c \to -\infty}f(c)=1$ when~$p_2 \neq 1/2$). 
Since~$G_{n,p_1}$ conditioned on having~$m$ edges has the same distribution as~$G_{n,m}$, 
by using \refL{lem:typical} to handle outcomes of~$G_{n,p_1}$ with an atypical number of edges~$e(G_{n,p_1})$ it follows~that
\begin{equation}\label{eq:thm:contain:m}
\begin{split}
\Bigabs{\Pr(G_{n,p_1} \isub G_{N,p_2}) - \! \sum_{m:\bigabs{m-\tbinom{n}{2}p_1} \le n^{4/3}} \!\!\!\!\!\!\!\!\!\!\!\! \Pr(G_{n,m} \isub G_{N,p_2}) \Pr(e(G_{n,p_1})=m)} 
 \: \le \: \Pr(G_{n,p_1} \not\in \cF_{n,p_1})=o(1) . 
\end{split}
\end{equation}
If~$\bigabs{m-\tbinom{n}{2}p_1} \le n^{4/3}$, 
then by using \refL{lem:typical} to handle the atypical event~$G_{n,m} \not\in \cT_{n,m}$ it follows that 
\begin{equation}\label{eq:thm:contain:Ha}
\begin{split}
\Bigabs{\Pr(G_{n,m} \isub G_{N,p_2}) -  \sum_{H \in \cT_{n,m}} \Pr(H \isub G_{N,p_2}) \Pr(G_{n,m}=H)} 
 \: \le \: \Pr(G_{n,m} \not\in \cT_{n,m}) = o(1) . 
\end{split}
\end{equation}
In particular, using the sharp threshold result \refL{lem:maincontain} for the event $H \isub G_{N,p_2}$ (which applies since~$|\psi \delta_m| =o(n)$ holds)
and \refL{lem:typical} for the typical event~$G_{n,m} \in \cT_{n,m}$, 
it follows that~\eqref{eq:thm:contain:Ha} implies 
\begin{equation}\label{eq:thm:contain:Hn}
\Pr(G_{n,m} \isub G_{N,p_2}) 
= \begin{cases} o(1) \quad & \text{if~$\psi \delta_m-c_N \le -\eps_N$,}\\ 
1-o(1) \quad & \text{if~$\psi \delta_m-c_N \ge \eps_N$.}\end{cases}
\end{equation}
If~$p_2=1/2$, 
then in the setting of~\eqref{eq:contain:12} we have~$\psi=0$ and~$\abs{c_N} \geq \eps_N$, so the sharp threshold result~\eqref{eq:contain:12}  follows by first inserting estimate~\eqref{eq:thm:contain:Hn} into~\eqref{eq:thm:contain:m} 
and finally using~\refL{lem:typical} to infer~$\Pr(G_{n,p_1} \in \cF_{n,p_1})=1-o(1)$. 

We focus on the case $p_2\neq1/2$ in the remainder of the proof.
Note that  estimate~\eqref{eq:thm:contain:Hn} does not apply when 
$|\psi\delta_{m}-c_N|< \eps_N$, but we shall now argue that we can effectively ignore this small range of edge-counts~$m$ in~\eqref{eq:thm:contain:m} above. 
Recall that~$n \to \infty$ as~$N \to \infty$. 
By the Central Limit Theorem it thus follows that, as~$N \to \infty$,  \begin{equation}\label{eq:eGnp:dto}
Z := \psi\delta_{e(G_{n,p_1})} = \psi\cdot \frac{e(G_{n,p_1})-\tbinom{n}{2}p_1}{n/2}
 \: \dto \: \Nor\bigpar{0, \; 2p_1(1-p_1)\psi^2} .
\end{equation}
Recalling that~$\eps_N = o(1)$ and $0 < \Var Z = 2p_1(1-p_1)\psi^2=\Theta(1)$, it follows~that
\begin{equation}\label{eq:eGnp:nomiddle}
\Pr(| Z-c_N| <  \eps_N) \: = \: o(1) .
\end{equation}
(This alternatively follows from~$\max_k\Pr(e(G_{n,p_1})=k) \le O(1)/\sqrt{n^2p_1(1-p_1)} \le O(1/n)$, say.)
Inserting estimates~\eqref{eq:thm:contain:Hn} and~\eqref{eq:eGnp:nomiddle} into~\eqref{eq:thm:contain:m}, using \refL{lem:typical} to infer~$\Pr(G_{n,p_1} \in \cF_{n,p_1})=1-o(1)$  
it follows~that 
\begin{equation}\label{eq:thm:contain}
\bigabs{\Pr(G_{n,p_1} \isub G_{N,p_2}) - \Pr(Z-c_N \ge 0)} \:  = \: o(1),
\end{equation}
which together with the convergence result~\eqref{eq:eGnp:dto} and~${\eps_N =o(1)}$
establishes the threshold result~\eqref{eq:contain:main}.
\end{proof}

In the following proof of~\refL{lem:maincontain} we shall apply the first and second moment method to the random variable~$X_H$, 
which we define as the number of induced copies of~$H$ in~$G_{N,p_2}$. 
Here the restriction to pseudorandom graphs~$H \in \cT_{n,m}$ will be key for controlling the expectation~$\E X_H$ and variance~$\Var X_H$.  
\begin{proof}[Proof of \refL{lem:maincontain} and Remark~\ref{rem:0-statement}]
Let~$H$ be any~$n$-vertex graph with~$e(H)=m$ edges (to clarify: with an eye on Remark~\ref{rem:0-statement} we initially do not assume~$H \in \cT_{n,m}$). 
Note that ${X_H=\sum_{S\in \binom{[N]}{n}} I_S}$, where~$I_S$ is the indicator random variable for the event that the induced subgraph~$G_{N,p_2}[S]$ is isomorphic to~$H$. 
Note that there are exactly $n!/\abs{\Aut(H)}$ distinct embeddings of~$H$ into~$S$. 
Using linearity of expectation and~$(N)_n=N!/(N-n)!$, 
in view of $m=\tbinom{n}{2}p_1+\delta_m n/2$, $a = {1/(p_2^{p_1}(1-p_2)^{1-p_1})}$, $p_2/(1-p_2) = a^{\psi}$ and~$(n-1)/2 = \log_a N + c_N/2$ it follows that the expected number of induced copies of~$H$ in $G_{N,p_2}$ satisfies
\begin{equation}\label{eq:firstmomentXH}
\begin{split}
\E X_H &= \sum_{S\in \binom{[N]}{n}} \E I_S = 
\binom{N}{n} \cdot \frac{n!}{\abs{\Aut(H)}}  \cdot p_2^{e(H)}(1-p_2)^{\binom{n}{2}-e(H)}\\
&= \frac{1}{\abs{\Aut(H)}} \cdot (N)_n \cdot p_2^m(1-p_2)^{\binom{n}{2}-m} \\
&= \frac{1}{\abs{\Aut(H)}} \cdot \lrsqpar{Ne^{O(n/N)} \cdot \lrpar{\frac{p_2}{1-p_2}}^{\delta_m/2}a^{-(n-1)/2}}^n \\
& =\frac{1}{\abs{\Aut(H)}} \cdot \lrsqpar{a^{\psi\delta_m-c_N+o(\eps_N)}}^{n/2} ,
\end{split}
\end{equation}
where we used~$n/N=\Theta((\log N)/N) = o(\eps_N)$ and~$1 < a = O(1)$ for the last step. 
If ${\psi\delta_m - c_N \leq - \eps_N}$, then using~${|\Aut(H)| \ge 1}$ and Markov's inequality together with~${\eps_N n \to \infty}$ as~${N \to \infty}$, we see that 
\[\Pr(X_H>0)  \le \E X_H \le a^{-\Omega(\eps_N n)} = o(1),
\]
establishing the~\mbox{$o(1)$ statement} in~\eqref{eq:p2} and also Remark~\ref{rem:0-statement} (since we so far did not assume~$H \in \cT_{n,m}$). 

In the remainder we focus on the~\mbox{$1-o(1)$ statement} in~\eqref{eq:p2}, i.e., we henceforth
fix~${H \in \cT_{n,m}}$ and assume that ${\psi\delta_m-c_N\geq\eps_N}$.
Since~${H \in \cT_{n,m} \subseteq \cA_N}$ implies~$\abs{\Aut(H)} = 1$, using estimate~\eqref{eq:firstmomentXH} we infer that 
\begin{equation}\label{eq:mu}
\mu := \E X_H = (N)_n \cdot p_2^m(1-p_2)^{\binom{n}{2}-m} = \lrsqpar{a^{\psi\delta_m-c_N+o(\eps_N)}}^{n/2} \ge a^{\Omega(\eps_N n)} \to\infty .
\end{equation}
To complete the proof of~\eqref{eq:p2}, 
using Chebyshev's inequality it thus suffices to show that ${\var X_H =o\bigpar{(\E X_H)^2}}$. 
Recall that ${X_H=\sum_{S\in \binom{[N]}{n}} I_S}$, where~$I_S$ is the indicator random variable for the event that the induced subgraph~$G_{N,p_2}[S]$ is isomorphic to~$H$. 
Since~$I_R$ and~$I_S$ are independent when~${|R \cap S| \le 1}$, we~have
\begin{equation}\label{eq:varXH}
\var X_H \: \le \: \sum_{2\le \ell \le  n}\underbrace{\sum_{R,S \in \binom{[N]}{n}: \abs{R\cap S}=\ell} \E(I_RI_S)}_{=: w_\ell} \: = \: \sum_{2\le \ell< n} \w_\ell + \underbrace{\E X_H}_{= \mu}.
\end{equation}
To bound the parameter~$\w_\ell$ defined in~\eqref{eq:varXH} for~$2 \le \ell < n$, we first note that there are $\binom{N}{n}\binom{n}{\ell}\binom{N-n}{n-\ell}$ ways of choosing $R,S\in\binom{[N]}{n}$ with $|R\cap S|=\ell$. 
To then get a handle on~$\E(I_RI_S)$, we count the number of ways we can embed two induced copies of~$H$ into~$R \cup S$ so that~$G_{N,p_2}[R]$ and~$G_{N,p_2}[S]$ are both isomorphic to~$H$: there are~$n!$ ways of embedding an induced copy of~$H$ into~$R$, at most~$\binom{n}{n-\ell} \cdot (n-\ell)! = \binom{n}{\ell} \cdot (n-\ell)! = (n)_{n-\ell}$ ways of choosing and embedding~$n-\ell$ vertices of a second induced copy of~$H$ into~$S \setminus R$, and at most $\max_{H\in \cT_{n,m},|L|=\ell} |\Aut(H[L])|$ ways of embedding the remaining~$\ell$ vertices of the second induced copy of~$H$ into~$R\cap S$, 
where the maximum is taken over all vertex-subsets $L \subseteq V(H)$ of size~$|L|=\ell$.
As these embeddings determine all edges and non-edges in~$G_{N,p_2}[R \cup S]$, it follows~that 
\begin{equation}\label{eq:generalineq0}
\begin{split}
	\w_\ell \le \binom{N}{n}\binom{n}{\ell} & \binom{N-n}{n-\ell}  \cdot n! \cdot  (n)_{n-\ell} \cdot \max_{H\in \cT_{n,m}, \: \abs{L}=\ell} |\Aut(H[L])| \\
	& \cdot \max_{H\in\cT_{n,m}, \: \abs{L}=\ell} p_2^{2m-e(H[L])} (1-p_2)^{2\binom{n}{2}-\binom{\ell}{2}-2m+e(H[L])} ,
\end{split}
\end{equation}
where the two maxima in~\eqref{eq:generalineq0} are each taken over all vertex-subsets $L \subseteq V(H)$ of size~$|L|=\ell$, as above. 
Note that~${n = (2+o(1))\log_a N}$, which implies $N-n > n-\ell$ for all sufficiently large~$N$. 
Recalling that $\mu = (N)_n p_2^m(1-p_2)^{\binom{n}{2}-m}$ by~\eqref{eq:mu}, using~$n^2/N = o(1)$ it follows that 
\begin{equation}\label{eq:generalineq1}
\frac{\w_\ell}{\mu^2} \leq
\underbrace{\frac{(N-n)_{n-\ell}}{(N)_n}}_{=(1+o(1))N^{-\ell}} \cdot  \binom{n}{\ell}^2 \cdot\max_{H \in \cT_{n,m}, \: \abs{L} = \ell}\abs{\Aut(H[L])} \cdot 
\underbrace{\max_{H\in \cT_{n,m}, \: \abs{L}=\ell} 
\biggpar{\frac{1-p_2}{p_2}}^{e(H[L])}(1-p_2)^{-\binom{\ell}{2}}}_{=:P_{\ell}}.
\end{equation}
Setting~$\epsp := 1/\max\bigcpar{-2\log_a(1-p_2),-2\log_a(p_2),2} >  0$, 
we now bound~$\w_\ell/\mu^2$ further using a case~distinction. 

\textit{\bf Case~$2\le \ell\le \epsp n$:} 
Here we use~$0 \leq e(H[L]) \leq \binom{\ell}{2}$ to deduce that the parameter~$P_\ell$ from~\eqref{eq:generalineq1} satisfies 
\begin{align}\label{eq:case1contain}
P_{\ell} \le \max_{0 \le k \le \binom{\ell}{2}} p_2^{-k}(1-p_2)^{-\binom{\ell}{2}+k} \le \biggsqpar{\max\biggcpar{\frac{1}{p_2},\frac{1}{1-p_2}}}^{\binom{\ell}{2}} .
\end{align}
Inserting this estimate and the trivial bound $ |\Aut(H[L])|\le |L|!$ into~\eqref{eq:generalineq1}, 
using~$\tbinom{n}{\ell} \le n^\ell/\ell!$ 
together with~$n = (2+o(1))\log_a N$ and $\ell \le \epsp n \le  (1+o(1))\log_{\max\{1/p_2,1/(1-p_2)\}} N$,
it follows for all  large enough~$N$ that 
\begin{equation}\label{eq:generalineq:case1}
\begin{split}
	\frac{\w_\ell}{\mu^2} &\leq 2N^{-\ell} \cdot \binom{n}{\ell}^2 \ell! \cdot P_{\ell} 
	\leq \Biggsqpar{N^{-1}n^{2} \cdot \biggsqpar{\max\biggcpar{\frac{1}{p_2},\frac{1}{1-p_2}}}^{(\ell-1)/2}}^{\ell} 
	\le N^{-\ell/3} \le a^{-\Omega(n)}.
\end{split}
\end{equation}

\textit{\bf Case~$\epsp n \le \ell < n$:} 
Here we improve the estimate~\eqref{eq:case1contain} for~$P_\ell$ using that~$H \in {\cT_{n,m} \subseteq \cE_{n,m}}$: this pseudorandom edge-property implies that, for any vertex-subset~$L\subseteq [n]$ of size  $|L|=\ell$, the induced number of edges satisfies~$e(H[L])=\tbinom{\ell}{2}m/\tbinom{n}{2} +  O(n^{2/3}(n-\ell))$.
Recalling~$a = {1/(p_2^{p_1}(1-p_2)^{1-p_1})}$, using~$m/\tbinom{n}{2}=p_1+\delta_m/n+O(1/n^2)$ and~$(1-p_2)/p_2 = a^{-\psi}$ it follows that the parameter~$P_\ell$ from~\eqref{eq:generalineq1} satisfies 
\begin{equation}\label{eq:case2contain}
	P_{\ell}
	= a^{\binom{\ell}{2}} \cdot \biggpar{\frac{1-p_2}{p_2}}^{\binom{\ell}{2} \delta_m/n + O(n^{2/3}(n-\ell))}
	=a^{(1-\psi \delta_m/n)\binom{\ell}{2} + O(\psi n^{2/3}(n-\ell))} .
\end{equation}
For~$\epsp n \le \ell < n - n^{2/3}$ we then estimate~$|\Aut(H[L])| \leq \ell! \le n^{\ell} \le e^{O(n^{2/3}(n-\ell))}$, and for~$n - n^{2/3} \le \ell < n$ we use~${H \in \cT_{n,m} \subseteq \cA_{n}}$ to obtain $\max_{|L|=\ell} |\Aut(H[L])|=1$.
Note that~$1 < a= O(1)$ and $\psi=O(1)$.
We now insert these estimates and~${\tbinom{n}{\ell}=\tbinom{n}{n-\ell} \le n^{n-\ell}}$ into~\eqref{eq:generalineq1}. 
Then, using $\ell-1 = (n-1) - (n-\ell)$ and  $\abs{\psi\delta_m/n} = o(1)$ as well as $(n-1)/2 = \log_a N + c_N/2$ and~$n-\ell \ge 1$, it follows for all large enough~$N$ that 
\begin{equation}\label{eq:generalineq:case2}
\begin{split}
	\frac{\w_\ell}{\mu^2}&\le 2N^{-\ell}n^{2(n-\ell)} \cdot a^{(1-\psi \delta_m/n)\binom{\ell}{2}} \cdot e^{O( n^{2/3}(n-\ell))} \\
	&\leq\lrsqpar{N^{-1}a^{(1-\psi \delta_m/n)(\ell-1)/2} \cdot e^{O(n^{2/3}(n-\ell)/\ell)}}^\ell\\
	&\leq\lrsqpar{N^{-1}a^{(n-1)/2 - \psi\delta_m/2 +o(1)  - (1-o(1))(n-\ell)/2} \cdot e^{o(n-\ell)}}^\ell\\
	& \le \Bigsqpar{a^{c_N -\psi\delta_m}}^{\ell/2} \le a^{-\Omega(\eps_N n)} \le n^{-\Omega(\log n)},
\end{split}
\end{equation}
where for the second last inequality we used that~$c_N-\psi\delta_m=-(\psi\delta_m-c_N)\leq-\eps_N$. 

To sum up: inserting the estimates~\eqref{eq:generalineq:case1} and~\eqref{eq:generalineq:case2} into the variance bound~\eqref{eq:varXH}, using~\eqref{eq:mu} it follows~that 
 \[   \frac{\var X_H }{(\E X_H)^2} \le  \frac{\sum_{2 \le \ell < n}\w_\ell+\mu}{\mu^2} \le o(1) + \frac{1}{\mu} =o(1),\]
 which completes the proof of \refL{lem:maincontain}, 
 as discussed.
\end{proof}

\begin{remark}\label{rem:maincontain}
If~$H \in \cT_{n,m}$ and~${\psi\delta_m-c_N \ge \eps_N}$, 
then the above proof of \refL{lem:maincontain} shows that 
\begin{equation*}\label{eq:secondmoment:refined}
\Pr\bigpar{|X_H-\E X_H| < \xi \E X_H} = 1-o(1)
\end{equation*}
for any~$\xi>0$, 
where the expected value~$\E X_H$ satisfies the asymptotic estimate~\eqref{eq:mu}. 
\end{remark}

\subsection{Asymptotic Distribution: Proof of \refT{thm:contain:dist}~\ref{enum:contain:dist:other}}\label{sec:contain:dist}
In this section we prove the distributional result \refT{thm:contain:dist}~\ref{enum:contain:dist:other}
as a corollary of the results from \refS{sec:proofcontain}, i.e., establish that for~$p_2 \neq 1/2$ the number of induced copies of~$G_{n,p_1}$ in~$G_{N,p_2}$ has a `squashed' log-normal limiting distribution.
Here the crux is~$X_H$ is strongly dependent on the number of edges: indeed, $X_H$ quickly changes from~$0$ to $(1+o(1)) \E X_H \to \infty$ as the number of edges~$m=e(H)$ passes through the threshold~${\psi\delta_m \sim c_N}$, see~Remarks~\ref{rem:0-statement} and~\ref{rem:maincontain}. 
This makes it plausible that~${\log(1+X_H)/\log N}$ changes abruptly around that threshold, which together with~$c_N \to c$ and the normal convergence result~\eqref{eq:eGnp:dto} for~$\psi\delta_{e(G_{n,p_1})}$
intuitively explains the form of the limiting distribution of~$\log(1+X_{G_{n,p_1}})/\log N$ in the convergence result~\eqref{eq:contain:dist:other}. 
\begin{proof}[Proof of \refT{thm:contain:dist}~\ref{enum:contain:dist:other}]
Note that~$c_N \to c$ and thus~$c_N = O(1)$ by assumption. 
In view of \refR{rem:maincontain}, define~$\cT_n$ as the union of all~$\cT_{n,m}$ from \refS{sec:proofcontain} with~$\bigabs{m-\tbinom{n}{2}p_1} \le n^{4/3}$. 
With analogous reasoning as for~\eqref{eq:thm:contain:m}--\eqref{eq:thm:contain:Ha}, 
using \refL{lem:typical} it follows that 
 \begin{equation}\label{eq:Gnp:Tn}
\Pr(G_{n,p_1} \not\in \cT_n) \: \le \: \Pr(G_{n,p_1} \not\in \cF_{n,p_1}) + \max_{m:\bigabs{m-\tbinom{n}{2}p_1} \le n^{4/3}}\!\!\!\!\!\!\Pr(G_{n,m} \not\in \cT_{n,m}) \: = \: o(1) .
\end{equation}
Furthermore, with an eye on the form~\eqref{eq:mu} of $\E X_H$ in~\refL{lem:maincontain}, 
as in~\eqref{eq:eGnp:dto} we have 
 \begin{equation}\label{eq:Z:dto}
Z := \psi \cdot \delta_{e(G_{n,p_1})} =\psi \cdot  \frac{e(G_{n,p_1})-\binom{n}{2}p_1}{n/2}  \:\: \dto \:\: \Nor\bigpar{0, \: 2p_1(1-p_1)\psi^2} .
\end{equation}

We now condition on~$G_{n,p_1} = H \in \cT_n$, so that~$X|_{G_{n,p_1}=H}=X_H$ and~$|Z| \le 2|\psi| n^{1/3} = o (\sqrt{\log N})$.
In particular, if~${Z-c_N \le -\eps_N}$ holds, then by applying  \refL{lem:maincontain} it follows that~whp
\begin{equation}\label{eq:log:small}
\frac{\log(1+X_H)}{\log N} = \frac{\log 1}{\log N} = 0. 
\end{equation}
Furthermore, if~${Z -c_N\ge \eps_N}$ holds, then by combining \refR{rem:maincontain}, estimate~\eqref{eq:mu} and~$\eps_N = o(1/\sqrt{\log N})$ with~${c_N=O(1)}$ and ${|Z| = o (\sqrt{\log N})}$ as well as~${n = 2 \log_a N + O(1)}$ and~${a=O(1)}$ it  follows that~whp
\begin{equation}\label{eq:log:large}
\frac{\log(1+X_H)}{\log N} = 
\frac{\log\bigpar{\E X_H}}{\log N} + \frac{O(1)}{\log N} =
\frac{n(Z-c_N)}{2\log_a N} + \frac{o(1)}{\sqrt{\log N}}
= Z-c_N + \frac{o(1)}{\sqrt{\log N}}.
\end{equation}
Finally, by combining estimates~\eqref{eq:Gnp:Tn} and~\eqref{eq:eGnp:nomiddle} with the conclusions of~\eqref{eq:log:small}--\eqref{eq:log:large},  
it follows~that 
\begin{equation*}\label{eq:log:approx}
\Pr\biggpar{\biggabs{\frac{\log(1+X)}{\log N}-\bigpar{Z-c_N}\indic{Z -c_N\ge \eps_N} } \ge \frac{1}{\sqrt{\log N}}} = o(1),
\end{equation*}
which together with~\eqref{eq:Z:dto} as well as~$c_N \to c$ and~$\eps_N \to 0$ establishes the convergence result~\eqref{eq:contain:dist:other}. 
\end{proof}

\begin{remark}\label{rem:sec:contain:dist}
Note that the above proof only uses the first and second moment method, 
i.e., does not require the asymptotics of~$\Var X_H$.
Given the somewhat complicated limiting distribution of~$X_H$, we leave it as an interesting open problem to complement~\eqref{eq:contain:dist:other} with near-optimal estimates on the rate of~convergence. 
Furthermore, as suggested out by a referee, it would also be worthwhile to investigate the limiting distribution of~$X_H$ in the case when~$n-(2\log_a N+1) \to \infty$. 
\end{remark}

\subsection{Asymptotic Poisson Distribution: Proof of \refT{thm:contain:dist}~\ref{enum:contain:dist:12}}\label{sec:contain:dist:poisson}
In this section we complete the proof of \refT{thm:contain:dist} by proving \refT{thm:contain:dist}~\ref{enum:contain:dist:12}, i.e., establishing that for~${p_2 =1/2}$ the number of induced copies of~$G_{n,p_1}$ in~$G_{N,p_2}$ has asymptotically Poisson distribution. 
To this end we shall use a version of the Stein-Chen method for Poisson approximation together with a two-round exposure argument 
and a refinement of the variance estimates from~\refS{sec:proofcontain}~for~${p_2=1/2}$.  
\begin{proof}[Proof of \refT{thm:contain:dist}~\ref{enum:contain:dist:12}]
Note that~$X=X_{n,N}$ conditional on~$G_{n,p_1}=H$ has the same distribution as~$X_H$. 
This enables us to again use a two-round exposure argument, where we first reveal~$G_{n,p_1}$ and then afterwards count the number of induced copies of~$G_{n,p_1}$ in~$G_{N,1/2}$.
To this end, let~$\cG_{n}$ be the set of all $n$-vertex graphs. 
Together with~$\cA_n \subseteq \cG_n$ as defined in \refS{sec:pseudorandom}, 
by applying \refL{lem:typical} it follows that 
\begin{align*}
\TV\bigpar{X, \: \Po(\mu)}
&= \sup_{S \subseteq \mathbb{N}}\bigabs{\pr(X\in S)-\pr(\Po(\mu)\in S)}\\ 
&= \sup_{S \subseteq \mathbb{N}}\Bigabs{\sum_{H\in \cG_n}\pr(G_{n,p_1}=H) \cdot \bigsqpar{\Pr(X_H\in S)-\Pr(\Po(\mu)\in S)}}  \\
&\le \sum_{H\in \cA_n}\pr(G_{n,p_1}=H) \cdot \sup_{S \subseteq \mathbb{N}}\bigabs{\Pr(X_H\in S)-\pr(\Po(\mu)\in S)} + \Pr(G_{n,p_1}\not \in \cA_n)\\ 
& \le \max_{H\in \cA_n}\TV\bigpar{X_H, \: \Po(\mu)} + o(1). 
\end{align*}

It thus remains to show that $\TV(X_H,\Po(\mu))=o(1)$ for any~$H \in \cA_n$. Fix a graph~$H \in \cA_n$. As in \refS{sec:proofcontain} we write~$X_H=\sum_{S\in \binom{[N]}{n}} I_S$, where~$I_S$ is the indicator random variable for the event that~$G_{N,p_2}[S]$ is isomorphic to~$H$. 
Since~$H \in \cA_n$ implies~$|\Aut(H)|=1$, using~\eqref{eq:firstmomentXH} with~$p_2=1/2$ it follows~that 
\begin{equation}\label{def:mu}
\E X_H = (N)_n \lrpar{\tfrac{1}{2}}^{\binom{n}{2}} = \mu \qquad \text{ for all~$H \in \cA_n$} .
\end{equation}
Note that $\TV(X_H,\Po(\mu))=o(1)$ immediately follows when~$\mu \to 0$  (since then~$X_H$ and~$\Po(\mu)$ are both whp zero). 
We may thus henceforth assume that~$\mu =\Omega(1)$, which in view of~$\mu =\bigpar{N \cdot 2^{-(n-1)/2+O(n/N)} }^{n}$ and~${n \ge 2\log_2 N-1+\eps_N}$ implies~$n=2\log_2N + O(1)$. 
Since~$I_R$ and~$I_S$ are independent when~${|R \cap S| \le 1}$, 
by applying the well-known version of the Stein-Chen method for Poisson approximation 
(based on so-called dependency graphs) 
stated in~\mbox{\cite[Theorem~6.23]{JLR}}  
it routinely follows~that
\begin{align}\label{def:Lambdai}
    \TV\bigpar{X_H, \: \Po(\mu)} \: \le \: 
   \underbrace{\min\bigcpar{\mu^{-1},1}}_{= O(\mu^{-1})} \cdot \Biggsqpar{ \underbrace{\: \sum_{\substack{R,S \in \binom{[N]}{n}:\\ 2 \le |R \cap S| < n}}\E(I_R I_S)}_{=:\Lambda_1} \ + \  \underbrace{\sum_{\substack{R,S \in \binom{[N]}{n}:\\ 2 \le |R \cap S| \le n}} \E I_R \E I_S}_{=: \Lambda_2} \: } . 
\end{align}
To establish the convergence result~\eqref{eq:contain:dist:12}, 
it thus remains to show that~$\Lambda_1$ and~$\Lambda_2$ are both~$o(\mu)$.

We first estimate~$\Lambda_2$ using basic counting arguments. 
In particular, note~that
\begin{equation*}
\label{eq:RStoS}
\Lambda_2 = \sum_{R \in \binom{[N]}{n}}\E I_R \sum_{\substack{S \in \binom{[N]}{n}:\\ 2 \le |R \cap S| \le n}} \E I_S 
\: \le \: \mu \cdot \sum_{2 \le k \le n} \binom{n}{k}\binom{N-n}{n-k} n!2^{-\binom{n}{2}}.
\end{equation*}
Recalling that~$n=2\log_2 N + O(1)$, 
for all~$2 \le k \leq n-1$ we see (with room to spare) that 
\[
\frac{\binom{n}{k}\binom{N-n}{n-k}}{\binom{n}{k+1}\binom{N-n}{n-k-1}}
= \frac{(k+1)(N-2n+k+1)}{(n-k)^2} > 1
\]
for all sufficiently large~$N$. 
Using~$n = 2\log_2 N + O(1)$ and~$n \geq 2\log_2N -1 +\eps_N$, it follows that
\begin{align*}
\frac{\Lambda_2}{\mu}
& \leq n \cdot \binom{n}{2}\binom{N-n}{n-2} \cdot n!2^{-\binom{n}{2}} \\
& \leq n^5 N^{n-2}  \cdot 2^{-\binom{n}{2}} \\
& \leq N^{-2} \cdot \Bigpar{n^{5/n} N  \cdot 2^{-(n-1)/2}}^n\\
& \leq O(2^{-n}) \cdot \Bigpar{2^{1-\eps_N/2 + o(\eps_N)}}^n\le 2^{-\Omega(\eps_Nn)} \leq n^{-\Omega(\log n)}. 
\end{align*}

Finally, we estimate~$\Lambda_1$ from~\eqref{def:Lambdai} 
by refining the variance estimates from the proof of Lemma~\ref{lem:maincontain}. 
Namely, bounding the following parameter~$\w_\ell$ from~\eqref{eq:varXH} as in~\eqref{eq:generalineq0}, 
by setting~$p_2=1/2$ we infer~that
\[
\w_\ell := \sum_{\substack{R,S \in \binom{[N]}{n}:\\ |R \cap S| =\ell}}\E(I_R I_S)
\leq \binom{N}{n}\binom{n}{\ell}\binom{N-n}{n-\ell} \cdot n! \cdot (n)_{n-\ell} \cdot \max_{H\in \cT_n,  \: \abs{L}=\ell} |\Aut(H[L])| \cdot 2^{\binom{\ell}{2}-2\binom{n}{2}} .\]
Using similar (but simpler) arguments as for inequality~\eqref{eq:generalineq1}, in view of~\eqref{def:mu} it then follows that 
\begin{equation}\label{eq:well:12}
\frac{w_\ell}{\mu} \leq N^{n-\ell} \cdot \binom{n}{\ell}^2 \cdot \max_{H \in\cT_n,  \: \abs{L}=\ell} \abs{\Aut(H[L])} \cdot 2^{\binom{\ell}{2}-\binom{n}{2}} .
\end{equation}
We now bound~$\w_\ell/\mu$ further using a case distinction.

\textit{\bf Case~$2 \leq \ell \le n - n^{2/3}$:} 
Here we insert the trivial bound $|\Aut(H[L])| \le \ell!$ into inequality~\eqref{eq:well:12}.
Writing~${\tbinom{\ell}{2}-\tbinom{n}{2}}= {-(n-\ell)(n+\ell-1)/2}$, 
using~$\tbinom{n}{\ell} \le n^\ell/\ell!$ and~${n \ge 2\log_2 N -1+\eps_N}$ it follows~that  
\begin{equation}\label{eq:ratio:12:small}
\begin{split}
\frac{\w_\ell}{\mu} 
&\leq N^{n-\ell} \cdot \binom{n}{\ell}^2 \cdot \ell! \cdot 2^{\binom{\ell}{2}-\binom{n}{2}} \\
& \leq \lrpar{N \cdot n^{2\ell/(n-\ell)} \cdot  2^{-(n-1+\ell)/2} }^{n-\ell}\\
& \leq \lrpar{2^{2-\eps_N - \ell(1-4(\log n)/(n-\ell))}}^{(n-\ell)/2} \leq 2^{-\Omega(\eps_N n)} \leq n^{-\Omega(\log n)},
\end{split}
\end{equation}
where for the second last inequality we optimized over all~$2 \leq \ell \le n - n^{2/3}$, using~$(\log n)/n = o(\eps_N)$.

\textit{\bf Case~$n - n^{2/3} \leq \ell < n$:} 
Here we exploit~$H \in \cA_n$ 
to insert $|\Aut(H[L])| =1$ into inequality~\eqref{eq:well:12}.
Writing again~${\tbinom{\ell}{2}-\tbinom{n}{2}}= {-(n-\ell)(n+\ell-1)/2}$, 
using~$\tbinom{n}{\ell}=\tbinom{n}{n-\ell} \le n^{n-\ell}$ 
and~$\ell-2+\eps_N=\Omega(n)$ it follows~that
\begin{equation}\label{eq:ratio:12:large}
\begin{split}
\frac{\w_\ell}{\mu} 
&\leq N^{n-\ell} \cdot \binom{n}{\ell}^2 \cdot 2^{\binom{\ell}{2}-\binom{n}{2}}\\
& \leq \lrpar{N \cdot n^2 \cdot 2^{-(n-1+\ell)/2}}^{n-\ell} \\
& \leq 
\lrpar{2^{-(\ell-2+\eps_N)+4 \log n}}^{(n-\ell)/2} \le 2^{-\Omega(n)}.
\end{split}
\end{equation}

To sum up: combining~\eqref{eq:ratio:12:small}--\eqref{eq:ratio:12:large} implies~$\Lambda_1=\sum_{2 \le \ell < n}\w_\ell=o(\mu)$, 
which completes the proof of~\eqref{eq:contain:dist:12}. 
\end{proof}

\begin{remark}\label{rem:sec:dist:poisson}
The above calculations yield~$\Var X_H =(1+o(1)) \E X_H$ by~\eqref{eq:varXH}, and can thus also be used to give an alternative (and compared to Section~\ref{sec:proofcontain} more direct) proof of the sharp threshold result~\refT{thm:contain}~\ref{enum:contain:12} for~$p_2=1/2$. 
We leave the optimal rate of convergence in~\eqref{eq:contain:dist:12} as an intriguing open problem. 
\end{remark}

\section{Maximum common induced subgraph problem}\label{sec:common}
In this section we prove Theorem~\ref{thm:maincommon}, i.e., establish two-point concentration of the size~$I_N$ of the maximum common induced subgraph of two independent random graphs~$G_{N, p_1}$ and~$G_{N,p_2}$ with constant edge-probabilities~${p_1,p_2 \in (0,1)}$. 
It naturally splits into two parts: 
we prove the whp upper bound~${I_N \le \floor{\nN+\eps_N}}$ in Section~\ref{sec:commonfirst}, 
and prove the whp lower bound ${I_N \ge \floor{\nN-\eps_N}}$ in Section~\ref{sec:commonsecond}.

To analyze the parameter~$n_N$ defined in~\eqref{def:nN}, it will be convenient to study the auxiliary~function
\begin{equation}\label{deg:gp}
g(p):=\max\bigcpar{\log b_0(p), \: 2\log b_1(p), \: 2\log b_2(p)} \qquad \text{for~$p \in [0,1]$,}
\end{equation}
where the functions~$b_0(p)$ and~$b_i(p)$ are defined as in~\eqref{def:b}. 
Note that~$g(p)$ 
depends on~$p_1,p_2 \in (0,1)$, but not on~$N$. 
The following key analytic lemma establishes lower bounds on $b_0,b_1$ and $b_2$, as well as properties of $g$ (the proof
is based on standard calculus techniques, and thus deferred to Appendix~\ref{sec:calcstuff}).
\begin{lemma}\label{lem:analytic}%
The functions~$g(p)$ and $b_j(p)$ with~$j \in \{0,1,2\}$ have the following properties:%
\begin{romenumerate}
    \item\label{item:approximate} The function~$g:[0,1] \to (0,\infty)$ has a unique minimizer~$p_0 \in (0,1)$.
    Furthermore, there is~$\xi=\xi(p_1,p_2) \in (0,1/2)$ such that $g(p)\ge  g(p_0) +\xi$ for all~${p\in [0,\xi]\cup [1-\xi,1]}$.
     \item\label{item:bounded} 
     We have $\min_{p\in [0,1]} b_i(p)\ge 1$ for $i\in \{1,2\}$.
     Furthermore, there is~$\lambda=\lambda(p_1,p_2)>1$ such that $\min_{p\in [0,1]} b_0(p) \ge \lambda$, 
		and for sufficiently large~$N \ge N_0(p_1,p_2)$  the following holds for any~${p=p(N) \in [0,1]}$ and $i\in \{1,2\}$: 
		if $x^{(i)}_N(p)\le x^{(0)}_N(p)$,  then~$b_i(p)\ge \lambda$.
     In particular, for any ${p=p(N)\in [0,1]}$ we~have 
		\begin{equation}\label{eq:firstorderofmin}
        \min\left\{\xot_N(p), \: \xone_N(p), \: \xtwo_N(p) \right\}= \frac{4\log N}{g(p)} + O(\log\log N) ,
    \end{equation}
		where the implicit constant in $O(\log\log N)$ depends only on~$p_1,p_2$ (and not on~$p$).
\end{romenumerate}\vspace{-0.125em}%
\end{lemma}%
\noindent
Note that, using the uniformity of estimate~\eqref{eq:firstorderofmin}, now~\eqref{def:nN:asymp} and~$n_N = \Theta(\log N)$ follow readily~from 
\begin{equation}\label{eq:firstorderonN}
    n_N =\max_{p\in [0,1]} \; \min\left\{\xot_N(p), \: \xone_N(p), \: \xtwo_N(p) \right\} \
    = \frac{ 4\log N}{g(p_0)} + O(\log\log N) . 
\end{equation}

\subsection{Upper bound: No common induced subgraph of size $\ceil{\nN+\eps_N}$} \label{sec:commonfirst}
In this section we prove the upper bound in~\eqref{eq:thmmain} of Theorem~\ref{thm:maincommon}.
More precisely, we show that whp there is no common induced subgraph of size $\ceil{\nN+\eps_N}$ in~$G_{N,p_1}$ and~$G_{N,p_2}$, 
which implies the desired whp upper bound~$I_N \le \ceil{\nN+\eps_N}-1 \le  \floor{\nN+\eps_N}$. 
Our proof-strategy employs a refinement of the standard first moment method: the idea is to apply different first moment bounds
for different densities of the potential common induced subgraphs, 
which in turn deal with the three containment bottlenecks discussed in Section~\ref{sec:heuristic:MCIS} (i.e., containment in~$G_{N,p_1}$, containment in~$G_{N,p_2}$, and containment in~both). 
As we shall see, these estimates are enabled by the corresponding three terms appearing in the definition~\eqref{def:nN}~of~$n_N$.

Turning to the details, to avoid clutter we define
\begin{equation}\label{eq:n:UB}
n:=\ceil{\nN+\eps_N}.
\end{equation}
Writing $\eve_m$ for the `bad' event that $G_{N,p_1}$ and $G_{N,p_2}$ have a common induced subgraph with $n$~vertices and $m$~edges,  
using a standard union bound argument it follows that
\begin{equation}\label{eq:In:UB}
\Pr(I_N \ge \ceil{\nN+\eps_N}) 
 \le \sum_{0\le m\le \binom{n}{2}} \hspace{-0.25em}\Pr(\eve_m).
\end{equation}
Since we are only interested in equality of common subgraphs up to isomorphisms, we define $\mathcal{G}_{n,m}$ as the set of all unlabeled graphs with $n$~vertices and $m$~edges. 
Let $\Xone_H$ and $\Xtwo_H$ denote the number of induced copies of $H$ in $G_{N,p_1}$ and $G_{N,p_2}$, respectively. 
The crux is that if $\eve_m$ holds, then the three inequalities~$\sum_{H\in \mathcal{G}_{n,m}} \Xone_H \ge 1$ and~$\sum_{H\in \mathcal{G}_{n,m}} \Xtwo_H \ge 1$ as well as~$\sum_{H\in \mathcal{G}_{n,m}} \Xone_H\Xtwo_H \ge 1$ all hold. 
Invoking Markov's inequality three times, it thus follows that the bad event~$\cB_m$ holds with probability at most
\begin{equation}\label{eq:fmmcommon}
\begin{split}
	\pr(\eve_m)& \le \min\Biggcpar{
\E\hspace{-0.25em}\sum_{H\in \mathcal{G}_{n,m}} \Xone_H, \ \
\E\hspace{-0.25em}\sum_{H\in \mathcal{G}_{n,m}} \Xtwo_H, \ \ 
\E\hspace{-0.25em}\sum_{H\in \mathcal{G}_{n,m}} \Xone_H\Xtwo_H} ,
\end{split}
\end{equation}
where the three expectations correspond to the three containment bottlenecks discussed in Section~\ref{sec:heuristic:MCIS}. 
%
Let 
\begin{equation}\label{def:density}
p=p(m,n):=\frac{m}{\binom{n}{2}}.
\end{equation}
The form of~\eqref{eq:fmmcommon} suggests that we might need a good estimate of~$|\mathcal{G}_{n,m}|$, but it turns out that we can avoid~this: 
by double-counting labeled graphs on $n$ vertices with $m$ edges, the crux is that we obtain the~identity
\begin{equation}
\label{eq:cgnmcardinality}
\sum_{H \in \cG_{n,m}} \frac{n!}{\abs{\Aut(H)}} = \binom{\binom{n}{2}}{m},
\end{equation}
which in view of~\eqref{eq:firstmomentXH} for~$G_{N,p_2}$ interacts favorably with the form of the expectations~$\E \Xone_H$ and~$\E\Xtwo_H$. 
We shall further approximate~\eqref{eq:cgnmcardinality} using the following consequence of Stirling's approximation~formula:
\begin{equation}\label{eq:entropyestimate}
\binom{\binom{n}{2}}{m} 
= e^{O(\log n)} \lrpar{\frac{1}{p^p(1-p)^{1-p}}}^{\binom{n}{2}} 
\qquad \text{when~$m=p\binom{n}{2}$.}
\end{equation}
The heuristic idea for bounding~$\pr(\eve_m)$ is to focus on the smallest expectation in~\eqref{eq:fmmcommon}  for~$m=p\binom{n}{2}$, 
but due to the definition~\eqref{def:nN} of the parameter~$n_N$ in~${n=\ceil{\nN+\eps_N}}$ it will be easier to use a case distinction depending on which term attains the minimum among~$\xot_N(p)$, $\xone_N(p)$ and~$\xtwo_N(p)$. 

\textbf{Case $x^{(i)}_N(p)=\min\bigcpar{\xot_N(p),\xone_N(p),\xtwo_N(p)}$ for $i\in\{1,2\}$:} 
Here we focus on 
on~$\E\sum_{H\in \mathcal{G}_{n,m}} X^{(i)}_H$ in~\eqref{eq:fmmcommon}.
Invoking the estimates~\eqref{eq:cgnmcardinality}--\eqref{eq:entropyestimate} to bound~$\sum_{H \in \cG_{n,m}}1/\abs{\Aut(H)}$, 
by using the definition~\eqref{def:b} of~$b_i=b_i(p) = (p/p_i)^p[(1-p)/(1-p_i)]^{1-p}$ and~$n^2/N=O((\log N)^2/N) = o(1)$ it follows similarly to~\eqref{eq:firstmomentXH} that 
\begin{equation}\label{eq:keyestimate:XHi:1}
 \begin{split}
     \pr(\eve_m)  & \le \E\sum_{H\in \mathcal{G}_{n,m}} X^{(i)}_H = \sum_{H\in \mathcal{G}_{n,m}} \E X^{(i)}_H\\
     & = \sum_{H \in \cG_{n,m}} \frac{1}{\abs{\Aut(H)}}
     \cdot (N)_n \cdot p_i^{p\binom{n}{2}}(1-p_i)^{(1-p)\binom{n}{2}} \\
     & = e^{O(\log n)} \frac{N^n}{n!} 
  b_i^{-\binom{n}{2}} .
 \end{split}
 \end{equation}
Since $x^{(i)}_N=x^{(i)}_N(p)\le \xot_N(p)$, by Lemma~\ref{lem:analytic}~\ref{item:bounded} we have~$b_i=b_i(p)\ge \lambda$ for some constant~${\lambda=\lambda(p_1,p_2)>1}$. 
Inserting Stirling's approximation formula~$n! = e^{O(\log n)}(n/e)^{n}$ and the identity~$eN = x^{(i)}_Nb_i^{(x^{(i)}_N-1)/2}$ from~\eqref{def:xi} into estimate~\eqref{eq:keyestimate:XHi:1}, 
using~$n \ge \nN+\eps_N =  x_N^{(i)} + \eps_N$ (which implies~$x_N^{(i)} \le n$ and $n-x_N^{(i)} \geq \eps_N$) it follows~that
\begin{equation}\label{eq:keyestimatebi1}
 \begin{split}
     \pr(\eve_m) & \le \lrsqpar{e^{O(\log n/n)}\tfrac{eN}{n}b_i^{-(n-1)/2}}^n
     = \lrsqpar{e^{O(\log n/n)}\tfrac{x^{(i)}_N}{n} b_i^{-(n-x^{(i)}_N)/2}}^n \le \lambda^{-\Omega(\eps_Nn)} \le n^{-\Omega(\log n)}.
 \end{split}
 \end{equation}
 
\textbf{Case~$x^{(0)}_N(p)=\min\bigcpar{\xot_N(p),\xone_N(p),\xtwo_N(p)}$:}
Here we focus 
on~$\E\sum_{H\in \mathcal{G}_{n,m}} \Xone_H\Xtwo_H$ in~\eqref{eq:fmmcommon}.
Exploiting independence of the two random graphs~$G_{N,p_1}$ and~$G_{N,p_2}$, using~$\abs{\Aut(H)}^2 \ge \abs{\Aut(H)}$, 
and applying the definition~\eqref{def:b} of~$b_0=b_0(p)$ it follows similarly to~\eqref{eq:keyestimate:XHi:1} that 
\begin{equation}\label{eq:keyestimate:XH1XH2:1}
 \begin{split}
    \pr(\eve_m) 
    & \le \E\sum_{H\in \mathcal{G}_{n,m}} \Xone_H\Xtwo_H = \sum_{H\in \mathcal{G}_{n,m}} \E\Xone_H \E\Xtwo_H \\
    & = 
    \sum_{H \in \cG_{n,m}} \frac{1}{\abs{\Aut(H)}^2}
    \cdot (N)_n^2 \cdot  \bigpar{p_1p_2}^{p\binom{n}{2}} \bigpar{(1-p_1)(1-p_2)}^{(1-p)\binom{n}{2}}\\
    & \leq e^{O(\log n)} \frac{N^{2n}}{n!} {b_0}^{-\binom{n}{2}} .
 \end{split}
 \end{equation}
By Lemma~\ref{lem:analytic}~\ref{item:bounded} we have~$b_0=b_0(p)\ge \lambda$ for some constant~${\lambda=\lambda(p_1,p_2)>1}$.
Recalling that~$\log N = \Theta(n)$ by~\eqref{eq:firstorderonN}, observe that the definition~\eqref{def:x_N} of~$x^{(0)}_N = x^{(0)}_N(p)$ ensures that
\begin{equation}\label{def:x0}
x^{(0)}_N b_0^{(x^{(0)}_N-1)/2} = \tfrac{x^{(0)}_N}{4 \log_{b_0}(N)} 
 \cdot e N^2 
 = \Bigpar{1+O\bigpar{\tfrac{\log \log N}{\log N}}} \cdot  eN^2 
= e^{O( (\log n)/n)} e N^2.
\end{equation}
Inserting Stirling's approximation formula~$n! = e^{O(\log n)}(n/e)^{n}$ and~\eqref{def:x0} into~\eqref{eq:keyestimate:XH1XH2:1}, 
using~$n \ge \nN+\eps_N = x_N^{(0)} + \eps_N$ (which implies~$x_N^{(0)} \le n$ and $n-x_N^{(0)} \geq \eps_N$) it follows~that
\begin{equation}\label{eq:mu0est}
\begin{split}
    \pr(\eve_m)& \le \lrsqpar{e^{O(\log n/n)}  \tfrac{e N^2}{n} {b_0}^{-(n-1)/{2}}  }^n 
=  \lrsqpar{e^{O(\log n/n)}\tfrac{x^{(0)}_N}{n} b_0^{-(n-x^{(0)}_N)/2}}^n \le \lambda^{-\Omega(\eps_Nn)} \le n^{-\Omega(\log n)} .
\end{split}
\end{equation}

To sum up: inserting~\eqref{eq:keyestimatebi1} and~\eqref{eq:mu0est} into~\eqref{eq:In:UB} readily gives~$\Pr(I_N \ge \ceil{\nN+\eps_N})=o(1)$, 
which as discussed completes the proof of the upper bound in~\eqref{eq:thmmain}, i.e., that whp~$I_N \le  \floor{\nN+\eps_N}$.

\subsection{Lower bound: Common induced subgraph of size $\floor{\nN-\eps_N}$} \label{sec:commonsecond}
In this section we prove the lower bound in~\eqref{eq:thmmain} of Theorem~\ref{thm:maincommon}. 
More precisely, we establish the desired whp lower bound~$I_N \ge \floor{\nN-\eps_N}$ by showing that whp there is a common induced subgraph of size $\floor{\nN-\eps_N}$ in~$G_{N,p_1}$ and~$G_{N,p_2}$.
Our proof-strategy is inspired by that of Lemma~\ref{lem:maincontain} from Section~\ref{sec:proofcontain}, though the technical details are significantly more involved: the idea is to pick an `optimal' edge-density~$p$, and then apply the second moment method to the total number of pairs of induced copies of~$H$ in $G_{N,p_1}$ and $G_{N,p_2}$, where we consider only  pseudorandom graphs~$H$ with $\floor{\nN-\eps_N}$ vertices and $\floor{p\binom{n}{2}}$ edges. 
Here the restriction to pseudorandom~$H$ will again be key for controlling the expectation and variance, 
with the extra wrinkle that the resulting involved variance calculations share some similarities with fourth moment~arguments 
(that require some new ideas to control the `overlaps' of different~$H$, including more careful enumeration~arguments).

Turning to the details, we pick~$p$ as a maximizer of $\min\bigcpar{ \xot_N(p), \xone_N(p), \xtwo_N(p)}$.
By comparing the asymptotic estimate~\eqref{eq:firstorderofmin} with the  asymptotics~\eqref{eq:firstorderonN} of~$n_N$, it follows from Lemma~\ref{lem:analytic} that there is a constant~$\xi=\xi(p_1,p_2) \in (0,1/2)$ such that~${p \in [\xi,1-\xi]}$ for all sufficiently large~$N$ (otherwise the first order asymptotics of~\eqref{eq:firstorderofmin} and~\eqref{eq:firstorderonN} would differ, contradicting our choice of~$p$).
To avoid clutter,  we~define
\begin{equation}\label{eq:n:LB:m}
n:=\floor{\nN-\eps_N} \qquad \text{and} \qquad m:=\floor{p\tbinom{n}{2}}.
\end{equation}
Recall that in Section~\ref{sec:proofcontain} we introduced the set~$\cT_{n,m} = \cA_n \cap \cE_{n,m}$ of pseudorandom graphs with~$n$~vertices and~$m$~edges, where~$\cA_n$ and~$\cE_{n,m}$ are defined as in \refS{sec:pseudorandom}. 
Since we are only interested in the existence of isomorphisms between induced subgraphs, we now define~$\T$ as the unlabeled variant of~$\cT_{n,m}$, which can formally be constructed by ignoring labels of the graphs in~$\cT_{n,m}$.
Since any graph in~$\cT_{n,m} \subseteq \cA_n$ is asymmetric, we have~$|\T| = |\cT_{n,m}|/n!$, and so it follows from Lemma~\ref{lem:typical} that 
\begin{equation}\label{eq:sizet}
    |\T|= (1+o(1)) \binom{\binom{n}{2}}{m}\frac{1}{n!}.
\end{equation}
As in Section~\ref{sec:commonfirst}, let $\Xone_H$ and $\Xtwo_H$ denote the number of induced copies of $H$ in $G_{N,p_1}$ and $G_{N,p_2}$, respectively. We then define the random variable 
\begin{equation}\label{eq:defX}
    X := \sum_{H\in \T} \Xone_H\Xtwo_H ,
\end{equation}
where~$X>0$ implies that $G_{N,p_1}$ and $G_{N,p_2}$ have a common induced subgraph on $n$ vertices, i.e., that~${I_N \ge n}$. 
To complete the proof of~$\Pr(I_N \ge \floor{\nN-\eps_N})=1-o(1)$, using Chebyshev's inequality it thus suffices to show that~$\E X \to \infty$ and $\Var X = o\bigpar{(\E X)^2}$ as~$N \to \infty$. 

We start by showing that~$\E X \to \infty$ as~$N \to \infty$.  
Analogous to Section~\ref{sec:proofcontain} we have ${\Xonetwo_H=\sum_{R_i\in \binom{[N]}{n}} \ionetwo_{H,R_i}}$, 
where~$\ionetwo_{H,R_i}$ is the indicator random variable for the event that the induced subgraph~$G_{N,p_i}[R_i]$ is isomorphic to~$H$.
Note that every unlabeled graph~$H \in \T$ satisfies the asymmetry property~$\cA_{n,m}$ from \refS{sec:pseudorandom}, so that~$|\Aut(H)|=1$.
It thus follows similarly to~\eqref{eq:keyestimate:XH1XH2:1} in Section~\ref{sec:commonfirst} that 
\begin{equation}\label{eq:mu:EX}
\mu :=\E X = \sum_{H\in \T} \E\Xone_H \E\Xtwo_H = |\T| \cdot \tbinom{N}{n}^2 \cdot  \mu_1\mu_2 ,
\end{equation}
where
\begin{equation}\label{eq:mu12}
\mu_i := \E \ionetwo_{H,R_i}=n!p_i^m(1-p_i)^{\binom{n}{2}-m}.
\end{equation}
Inspecting the form of~\eqref{eq:mu:EX} and~\eqref{eq:cgnmcardinality}, 
note that the asymptotic estimate~\eqref{eq:sizet} of~$|\T|$ allows us to estimate~$\E X$ analogously to~$\E\sum_{H \in \mathcal{G}_{n,m}}\Xone_H\Xtwo_H$ in~\eqref{eq:keyestimate:XH1XH2:1}--\eqref{eq:mu0est}. 
Indeed, using~${b_0=b_0(p) \ge \lambda=\lambda(p_1,p_2) > 1}$ and~${n \le \nN-\eps_N \le x_N^{(0)} - \eps_N}$ (which implies~$x_N^{(0)} \ge n$ and $x_N^{(0)}-n \geq \eps_N$) it here follows~that
\begin{equation}\label{eq:mu:EX:asymp}
\begin{split}
\E X &=  (1+o(1)) \binom{\binom{n}{2}}{m}\frac{1}{n!} \cdot (N)_n^2  \bigpar{p_1p_2}^{p\binom{n}{2}} \bigpar{(1-p_1)(1-p_2)}^{(1-p)\binom{n}{2}}\\
    & = e^{O(\log n)} \frac{N^{2n}}{n!} {b_0}^{-\binom{n}{2}} =
\lrsqpar{e^{O(\log n/n)}\tfrac{x^{(0)}_N}{n} b_0^{(x^{(0)}_N-n)/2}}^n  \geq {\lambda}^{\Omega(\eps_N n)} \geq n^{\Omega(\log n)}  \to \infty . 
\end{split}
\end{equation}

The remainder of this section is devoted  to showing that $\Var X = o\bigpar{(\E X)^2}$. 
Note that~$I^{(1)}_{H,R_1}$ and~$I^{(2)}_{H',R_2}$ are always independent. 
Since~$I^{(i)}_{H,R_i}$ and $I^{(i)}_{H',S_i}$ are independent when~$\abs{R_i \cap S_i} \leq 1$, it follows~that 
\begin{equation}\label{eq:VarX}
\var X \leq \sum_{\substack{0 \le \ell_1,\ell_2 \le n:\\ \max\{\ell_1,\ell_2\}\geq 2}}\underbrace{\sum_{H,H'\in \T} \Biggpar{\sum_{\substack{R_1,S_1 \in \binom{[N]}{n}:\\ |R_1\cap S_1|=\ell_1}} \E \ione_{H,R_1}\ione_{H',S_1}}\biggpar{\sum_{\substack{R_2,S_2 \in \binom{[N]}{n}:\\ |R_2\cap S_2|=\ell_2}}\E \itwo_{H,R_2}\itwo_{H',S_2}}}_{=: w_{\ell_1,\ell_2}} .
\end{equation}
Our upcoming estimates of this somewhat elaborate variance expression use a case distinction depending on whether~$\ell_1\ge \max\{\ell_2,2\}$ or~$\ell_2\ge \max\{\ell_1,2\}$ holds. 
Both cases can be handled by the same argument (with the roles of~$G_{N,p_1}$ and~$G_{N,p_2}$ in the below definition~\eqref{eq:VarX:well} of~$u_{\ell_1}$ and~$v_{\ell_2}$ interchanged), so we shall henceforth focus on the case where~$\ell_1\ge \max\{\ell_2,2\}$ holds. 
In this case we find it convenient to~estimate
\begin{equation}\label{eq:VarX:well}
\frac{w_{\ell_1,\ell_2}}{\mu^2}
\le \underbrace{\frac{\sum_{H,H'\in \T}\sum_{R_1,S_1:|R_1\cap S_1|=\ell_1} \E \ione_{H,R_1}\ione_{H',S_1}}{|\T|^2\binom{N}{n}^2\mu_1^2}}_{=:\u_{\ell_1}} \cdot \underbrace{\max_{H,H'\in \T} \frac{\sum_{R_2,S_2:|R_2\cap S_2|=\ell_2}\E \itwo_{H,R_2}\itwo_{H',S_2}}{\binom{N}{n}^2\mu_2^2}}_{=:\v_{\ell_2}},
\end{equation}
where the sums are taken over all vertex sets $R_1,S_1 \in \tbinom{[N]}{n}$ and $R_2,S_2 \in \binom{[N]}{n}$, as in~\eqref{eq:VarX} above.
This splitting allows us to deal with one random graph~$G_{N,p_i}$ at a time, 
which we shall exploit when bounding~$u_{\ell_1}$ and~$v_{\ell_2}$ in the upcoming Sections~\ref{sec:vell}--\ref{sec:uell}. 
For later reference, we now define the (sufficiently small)~constant
\begin{equation}\label{eq:def:zeta}
\zeta := \min\Biggcpar{\frac{\log b_0}{4\log \max\Bigcpar{\frac{1}{p_1},\frac{1}{1-p_1},\frac{1}{p_2},\frac{1}{1-p_2}}}, \: \frac{1}{2}}.
\end{equation}

\subsubsection{Contribution of~$G_{N,p_2}$ to variance: Bounding $\v_{\ell_2}$}\label{sec:vell}
We first bound the parameter~$v_{\ell_2}$ defined in~\eqref{eq:VarX:well} for~$0 \le \ell_2 \le n$, using similar arguments as for the variance calculations from Section~\ref{sec:proofcontain}. 
To avoid clutter, in the following we will drop the index from~$\ell_2$ and simply write~$\ell=\ell_2$ for brevity. 
We start with the pathological case~$0 \le \ell \le 1$, where in view of~\eqref{eq:mu12} we trivially (due to independence of~$\itwo_{H,R_2}$ and~$\itwo_{H',S_2}$ when~$|R_2\cap S_2| \le 1$) have
\begin{equation}\label{eq:vell:patho}
v_{\ell} = \max_{H,H'\in \T} \frac{\sum_{R_2,S_2:|R_2\cap S_2|=\ell}\E \itwo_{H,R_2} \E \itwo_{H',S_2}}{\binom{N}{n}^2\mu_2^2} \le 1 \qquad \text{when~$0 \le \ell \le 1$.}
\end{equation}

It remains to bound~$v_{\ell}$ for~$2 \le \ell \le n$.
Here we shall reuse some ideas from  Section~\ref{sec:proofcontain}: 
bounding the numerator in the definition~\eqref{eq:VarX:well} of~$v_{\ell}$ as in~\eqref{eq:generalineq0}--\eqref{eq:generalineq1},
in view of~$\mu_2={n!p_2^m(1-p_2)^{\binom{n}{2}-m}}$ it follows that  
\begin{equation}\label{eq:bound:vell}
\begin{split}
\v_\ell &\le \frac{\binom{N}{n}\binom{n}{\ell}\binom{N-n}{n-\ell} \cdot n! \cdot (n)_{n-\ell}}{\binom{N}{n}^2(n!)^2} \cdot \max_{H\in \T, \: |L|=\ell} |\Aut(H[L])| \\
& \phantom{\le \frac{\binom{N}{n}\binom{n}{\ell}\binom{N-n}{n-\ell} \cdot n! \cdot (n)_{n-\ell}}{\binom{N}{n}^2(n!)^2}} \cdot \max_{\substack{H\in \T, \: |L|=\ell}}\frac{p_2^{2m-e(H[L])} (1-p_2)^{2\binom{n}{2}-2m-\binom{\ell}{2}+e(H[L])}}{ p_2^{2m} (1-p_2)^{2\binom{n}{2}-2m}}\\
&\le \underbrace{\frac{(N-n)_{n-\ell}}{(N)_n}}_{=(1+o(1))N^{-\ell}} \cdot  \binom{n}{\ell}^2
\cdot 
\max_{H\in \T, \: |L|=\ell} |\Aut(H[L])|
\cdot \underbrace{\max_{H\in \T, \: |L|=\ell}{p_2^{-e(H[L])} (1-p_2)^{-\binom{\ell}{2}+e(H[L])}}}_{=:P_{\ell,2}},
\end{split}
\end{equation}
where the four maxima in~\eqref{eq:bound:vell} are each taken over all vertex-subsets $L \subseteq V(H)$ of size~$|L|=\ell$, as before. 
We define the parameter~$P_{\ell,1}$  analogously to~$P_{\ell,2}$ (by replacing~$p_2$ with~$p_1$).
Since every unlabeled graph ${H \in \T}$ satisfies the pseudorandom edge-properties of~$\cE_{n,m}$ from Section~\ref{sec:pseudorandom}, 
it follows that 
\begin{equation}\label{eq:gbound:0}
\begin{split}
P_{\ell,i}
& \le \max_{k: \abs{k-\binom{\ell}{2}m/\binom{n}{2}}\le n^{2/3}(n-\ell)} p_i^{-k} (1-p_i)^{-\binom{\ell}{2}+k}
\end{split}
\end{equation}
for any~$0 \le \ell \le n$, which in view of~$m=\floor{p\tbinom{n}{2}}$ 
yields, similarly to~\eqref{eq:case1contain} and~\eqref{eq:case2contain} from Section~\ref{sec:proofcontain}, that 
\begin{equation}\label{eq:gbound}
\begin{split}
P_{\ell,i}
& \le \begin{cases}
        \biggsqpar{\max\biggcpar{\frac{1}{p_i},\frac{1}{1-p_i}}}^{\binom{\ell}{2}} \quad & \text{if~$0 \le \ell\le \zeta n$,} \\
        e^{O\lrpar{n^{2/3}(n-\ell)+1}} \cdot \lrpar{{p_i^p (1-p_i)^{1-p}} }^{-\binom{\ell}{2}} \quad 
        & \text{if~$\zeta n\le \ell\le n$.}
    \end{cases}
\end{split}
\end{equation}
After these preparations, we are now ready to bound~$\v_\ell$ further using a case distinction, 
where~$\zeta$ is defined as in~\eqref{eq:def:zeta}.

\textbf{Case $2\le \ell\le \zeta n$:} 
Here we proceed similarly to~\eqref{eq:case1contain}--\eqref{eq:generalineq:case1}, and exploit that for all large enough~$N$ we~have 
\begin{equation}\label{eq:bound:ell}
\ell \leq \zeta n \leq \zeta \cdot \nN \le \zeta \cdot x^{(0)}_N(p) \leq \zeta \cdot \frac{4\log N}{\log b_0} \leq \frac{\log N}{\log \max\Bigcpar{\frac{1}{p_1},\frac{1}{1-p_1},\frac{1}{p_2},\frac{1}{1-p_2}}}
\end{equation}
by our choice of~$\zeta$.
Inserting~\eqref{eq:gbound} and the trivial bound~$\abs{\Aut(H[L])} \leq \abs{L}! = \ell!$ into~\eqref{eq:bound:vell}, 
using~$\binom{n}{\ell} \le n^\ell/\ell!$ and~\eqref{eq:bound:ell} it follows for all large enough~$N$ that
\begin{align}\label{eq:cestimate1}
\v_\ell& \leq 2 N^{-\ell} \cdot \binom{n}{\ell}^2
\ell! \cdot P_{\ell,2} 
	\leq \Biggpar{N^{-1}n^{2} \cdot  \biggsqpar{\max\biggcpar{\frac{1}{p_2},\frac{1}{1-p_2}}}^{(\ell-1)/2}}^{\ell} 
 \le N^{-\ell/3} \le e^{-\Omega(n)}.
\end{align}

\textbf{Case $\zeta n\le \ell\le  n$:} 
Here we refine the previous argument, 
proceeding similarly to~\eqref{eq:generalineq:case2}. 
For~$\zeta n \le \ell < n - n^{2/3}$ we estimate~$|\Aut(H[L])| \leq \ell! \le n^{\ell} \le e^{O(n^{2/3}(n-\ell))}$, and for~$n - n^{2/3} \le \ell \le n$ 
we have~$|\Aut(H[L])|=1$ since 
every unlabeled graph~$H \in \T$ satisfies the asymmetry property~$\cA_{n,m}$ from \refS{sec:pseudorandom}. 
Inserting these estimates and~$\binom{n}{\ell}=\binom{n}{n-\ell} \le n^{n-\ell} = e^{O(n^{2/3}(n-\ell))}$ into~\eqref{eq:bound:vell}, using~\eqref{eq:gbound} it  follows~that 
\begin{equation}\label{eq:cestimate4}
\begin{split}
\v_\ell &\leq 2 N^{-\ell} \cdot \binom{n}{\ell}^2 \cdot e^{O\lrpar{n^{2/3}(n-\ell)+1}} \cdot P_{\ell,2} 
\leq O(1) \cdot \Bigsqpar{N^{-1}(p_2^p(1-p_2)^{1-p})^{-(\ell-1)/2} e^{O\lrpar{n^{-1/3}(n-\ell)}}}^\ell,
\end{split}
\end{equation}
which in \refS{sec:well} will later turn out to be a useful upper bound.

\subsubsection{Contribution of~$G_{N,p_1}$ to variance: Bounding $\u_{\ell_1}$}\label{sec:uell}
We now bound the parameter~$u_{\ell_1}$ defined in~\eqref{eq:VarX:well} for~$2 \le \ell_1 \le n$. 
To avoid clutter, in the following we will drop the index from~$\ell_1$ and simply write~$\ell=\ell_1$ for brevity. 
For the simpler case $2\le \ell\le \zeta n$ we shall reuse the argument from Section~\ref{sec:vell}, while the more elaborate case $\zeta n\le \ell\le n$ requires further new~ideas. 
Recall that~$\zeta$ is defined as in~\eqref{eq:def:zeta}.  

\textbf{Case  $2\le \ell\le \zeta n$:}
Using the pair $H,H'\in \T$ which maximizes the summand as an upper bound reduces this case to the analogous bounds for~$\v_{\ell}$ 
from Section~\ref{sec:vell}. Indeed, by proceeding this way we obtain that 
\begin{equation}\label{eq:bestimate1}
    \u_\ell \le \max_{H,H'\in \T} \frac{\sum_{R_1,S_1:|R_1\cap S_1|=\ell}\E \ione_{H,R_1}\ione_{H',S_1}}{\binom{N}{n}^2\mu_1^2}\le e^{-\Omega(n)},
\end{equation}
where the last inequality follows word-by-word (with $p_2$ replaced by $p_1$, and~$P_{\ell,2}$ replaced by~$P_{\ell,1}$)
from the arguments leading to~\eqref{eq:bound:ell}--\eqref{eq:cestimate1}.

\textbf{Case $\zeta n \le \ell \leq n$:} 
Here one key idea is to interchange the order of summation in the numerator to~obtain
\begin{equation}\label{eq:bound:uell}
\u_{\ell} = 
\frac{\sum_{R_1,S_1:|R_1\cap S_1|=\ell}\sum_{H,H'\in \T}\E \ione_{H,R_1}\ione_{H',S_1}}{\binom{N}{n}^2|\T|^2\mu_1^2} ,
\end{equation}
which then allows us to exploit that not too many choices of the pseudorandom graphs~$H,H' \in \T$ can intersect in a `compatible' way. 
Indeed, if we proceeded similarly to the argument leading to~\eqref{eq:bound:vell}, \eqref{eq:cestimate1} and~\eqref{eq:bestimate1}, 
then after choosing~${R_1,S_1 \in \binom{[N]}{n}}$ with $\abs{R_1 \cap S_1} = \ell$, 
in the numerator of~\eqref{eq:bound:uell} we would use
\begin{equation*}
|\T|^2 \cdot n! \cdot (n)_{n-\ell} \cdot \max_{H\in \T,  \: |L|=\ell} |\Aut(H[L])| \: \le \:  |\T|^2 (n!)^2 
\end{equation*}
as a simple upper bound  on the number~$\Lambda_{\ell}(R_1,S_1)$ of labeled graphs~$F$ on~$R_1 \cup S_1$ such that~$F[R_1]$ and~$F[S_1]$ are isomorphic to some~$H$ and~$H'$, respectively, where~$H,H' \in \T$. 
Here our key improvement idea is to more carefully enumerate all possible such graphs~$F$, by first choosing the edges in the intersection~$R_1 \cap S_1$ of size~$|R_1 \cap S_1|=\ell$, and only then the remaining edges of~$F$. 
The crux is that since all possible~$H,H' \in \T$ satisfy the pseudorandom edge-properties of~$\cE_{n,m}$ from Section~\ref{sec:pseudorandom}, 
we know in advance that all possible numbers~$k$ of edges inside~$R_1 \cap S_1$ must satisfy $\bigabs{k-\binom{\ell}{2}m/\binom{n}{2}} \le n^{2/3}(n-\ell)$. 
Hence, by first choosing~$k$ edges in the intersection~$R_1 \cap S_1$, and then the remaining~$2(m-k)$ edges, 
using~$m=\floor{p\binom{n}{2}}$ it follows~that 
\begin{equation}\label{eq:bound:Lambda}
\begin{split}
\Lambda_{\ell}(R_1,S_1) 
&\le \sum_{k: \abs{k-\binom{\ell}{2}m/\binom{n}{2}}\le n^{2/3}(n-\ell)} \binom{\binom{\ell}{2}}{k} \binom{2\bigsqpar{\binom{n}{2}-\binom{\ell}{2}}}{2(m-k)}\\
&\le n^2 \cdot \binom{\binom{\ell}{2}}{\floor{p\binom{\ell}{2}}}\binom{2\bigsqpar{\binom{n}{2}-\binom{\ell}{2}}}{\floor{2p\bigsqpar{\binom{n}{2}-\binom{\ell}{2}}}}  \cdot n^{O\lrpar{n^{2/3}(n-\ell)+1}} \\
&\le \lrpar{{p^p(1-p)^{1-p}}}^{-2\binom{n}{2}+\binom{\ell}{2}} \cdot e^{O\lrpar{\log n+n^{3/4}(n-\ell)}} ,
\end{split}
\end{equation}
where for the last inequality we used Stirling's approximation formula similarly to~\eqref{eq:entropyestimate}. 
Using again Stirling's approximation formula similarly to~\eqref{eq:entropyestimate}, from the asymptotic estimate~\eqref{eq:sizet} for~$|\T|$ it follows~that 
\begin{equation}\label{eq:estimate:T}
|\T|^2(n!)^2= (1+o(1)) \binom{\binom{n}{2}}{m}^2 = e^{O(\log n)} \lrpar{{p^p(1-p)^{1-p}}}^{-2\binom{n}{2}},
\end{equation}
which in view of~$\binom{\ell}{2}=\Theta(n^2)$ makes it transparent that the refined upper bound~\eqref{eq:bound:Lambda} on~$\Lambda_{\ell}(R_1,S_1)$ is significantly smaller than the simple upper bound~$|\T|^2(n!)^2$ mentioned above.  
After these preparations, we are now ready to estimate~$\u_\ell$ as written in~\eqref{eq:bound:uell}: namely, (i)~we bound the numerator as in~\eqref{eq:bound:vell}, the key difference being that we use~$\Lambda_{\ell}(R_1,S_1)$ to account for the choices of ${H,H'\in \T}$ and their embeddings into~$R_1$ and~$S_1$, and (ii)~we also use~\eqref{eq:estimate:T} to estimate the denominator in~\eqref{eq:bound:uell}. 
Taking these differences to~\eqref{eq:bound:vell} into account, using~$\mu_1={n!p_1^m(1-p_1)^{\binom{n}{2}-m}}$ it then follows that  
\begin{equation*}
\begin{split}
\u_{\ell} & \le \frac{\binom{N}{n}\binom{n}{\ell}\binom{N-n}{n-\ell}}{\binom{N}{n}^2|\T|^2(n!)^2}\cdot 
\max_{R_1,S_1:|R_1\cap S_1|=\ell}\Lambda_{\ell}(R_1,S_1) 
\cdot \max_{\substack{H\in \T, \: |L|=\ell}}\frac{p_1^{2m-e(H[L])} (1-p_1)^{2\binom{n}{2}-2m-\binom{\ell}{2}+e(H[L])}}{ p_1^{2m} (1-p_1)^{2\binom{n}{2}-2m}}\\
&\le (1+o(1))N^{-\ell} \cdot  \binom{n}{\ell}^2\ell! \cdot \lrpar{{p^p(1-p)^{1-p}}}^{\binom{\ell}{2}} e^{O\lrpar{\log n+n^{3/4}(n-\ell)}}  
\cdot P_{\ell,1} ,
\end{split}
\end{equation*}
where~$P_{\ell,1}$ is defined as below~\eqref{eq:bound:vell}. 
Combining the upper bound~\eqref{eq:gbound:0}--\eqref{eq:gbound} for~$P_{\ell,1}$ with the definition~\eqref{def:b} of~${b_1=b_1(p)} = {(p/p_1)^p[(1-p)/(1-p_1)]^{1-p}}$,
using Stirling's approximation formula~$\ell! = e^{O(\log n)}(\ell/e)^{\ell}$ and~$\binom{n}{\ell}^2=\binom{n}{n-\ell}^2\le n^{2(n-\ell)}=e^{O(n^{3/4}(n-\ell))}$ it follows that  
\begin{align}\label{eq:bound:uell:prep}
    u_\ell 
 \le \lrpar{\tfrac{\ell}{eN}}^{\ell} \cdot b_1^{\binom{\ell}{2}} \cdot e^{O\lrpar{\log n+n^{3/4}(n-\ell)}} 
    \le \lrsqpar{e^{O\lrpar{\log n/n+n^{-1/4}(n-\ell)}} \tfrac{\ell}{eN} b_1^{(\ell-1)/2}  }^\ell .
\end{align}
We now find it convenient to treat the case where~$b_1$ is close to~$1$ separately from the case where~$b_1$ is bounded away from~$1$. 
From~\eqref{eq:firstorderonN} we know that~$\ell\le n=O(\log N)$, so there exists a constant $\rho=\rho(p_1,p_2)>0$ such that $b_1\le 1+\rho$ implies $b_1^{(\ell-1)/2}\le e^{\rho(\ell-1)/2}\le \sqrt{N}$.
In case of~$b_1\le 1+\rho$ we thus~infer~that 
\begin{equation}\label{eq:322b1small}
    u_\ell \le 
    \lrsqpar{e^{O(n^{3/4})} \tfrac{\ell}{eN} b_1^{(\ell-1)/2}  }^\ell\le \lrsqpar{ e^{O(n^{3/4})} \tfrac{\ell}{\sqrt{N}}}^\ell\le n^{-\Omega(\log n)}.
\end{equation} 
In case of~$b_1\ge 1+\rho$ we insert the identity~$eN = x^{(1)}_Nb_1^{(x^{(1)}_N-1)/2}$ from~\eqref{def:xi} 
into estimate~\eqref{eq:bound:uell:prep}, 
and then use ${\zeta n \le \ell\le n} \le {n_N-\eps_N} \le {\xone_N- \eps_N}$  
(which implies~$|n-\ell| \le \xone_N-\ell$ and~$\ell\le \xone_N$ as well as~$\xone_N-\ell \geq \eps_N$ and~$\ell \ge \zeta n$)
 to infer~that 
\begin{equation}\label{eq:322b1large}
u_\ell \le  \lrsqpar{e^{O(\log n/n)+o(n-\ell)} \tfrac{\ell}{x^{(1)}_N} b_1^{-(x^{(1)}_N-\ell)/2}  }^\ell
\le  \lrsqpar{e^{O(\log n/n)}  (1+\rho)^{-(x^{(1)}_N-\ell)/3} }^\ell 
\le e^{-\Omega(\eps_Nn)}\le n^{-\Omega(\log n)} .
\end{equation}
To sum up: from the two estimates~\eqref{eq:322b1small}--\eqref{eq:322b1large} it  follows that we always have 
\begin{equation}\label{eq:bestimate2}
 \u_{\ell} \leq n^{-\Omega(\log n)} \qquad \text{when~$\zeta n \le \ell \le n$.}
 \end{equation}

\subsubsection{Putting things together: Bounding $w_{\ell_1,\ell_2}/\mu^2$}\label{sec:well}
Using the estimates from Sections~\ref{sec:vell}--\ref{sec:uell}, 
in the case~$\ell_1\ge \max\{\ell_2,2\}$ we are now ready to bound the parameter~$w_{\ell_1,\ell_2}/\mu^2 \le \u_{\ell_1}\v_{\ell_2}$ from~\eqref{eq:VarX:well} using a case distinction based on whether or not $\ell_2\le \zeta n$. 

\textbf{Case~$0 \le \ell_2\le \zeta n$ and~$2 \le \ell_1 \le n$:} 
From~\eqref{eq:vell:patho} and~\eqref{eq:cestimate1}, we infer that $v_{\ell_2}\le 1$ for all $0 \le \ell_2\le \zeta n$. 
From~\eqref{eq:bestimate1} and~\eqref{eq:bestimate2}, we infer that $u_{\ell_1}\le n^{-\Omega(\log n)}$ for all~$2 \le \ell_1 \le n$. Hence
\begin{equation}\label{eq:bound:uv:1}
\u_{\ell_1}\v_{\ell_2} \le n^{-\Omega(\log n)}.
\end{equation}

\textbf{Case~$\zeta n\le \ell_2 \le \ell_1\le n$:}.
Here we split our analysis into two cases, based on whether or not the upper bound~\eqref{eq:cestimate4} on $v_{\ell_2}$ is effective on its own. 
We start with the case where~$N(p_2^p(1-p_2)^{1-p})^{
(\ell_2-1)/2}\ge e^{n^{3/4}}$ holds: here estimate~\eqref{eq:cestimate4} implies~$\v_{\ell_2}\le e^{-\Omega(\ell_2n^{3/4})}=o(1)$, which together with~\eqref{eq:bestimate2} 
yields that
\begin{equation}\label{eq:bound:uv:2}
\u_{\ell_1}\v_{\ell_2}
\leq n^{-\Omega(\log n)}.
\end{equation}
We henceforth consider the remaining case, where in view of~$\ell_1 = \Omega(n)$ we have 
\begin{equation*}
N(p_2^p(1-p_2)^{1-p})^{
(\ell_1+\ell_2-1)/2}\le e^{n^{3/4}} \cdot (p_2^p(1-p_2)^{1-p})^{\ell_1/2}\le  e^{-\Omega(n)} .
\end{equation*}
By combining the estimates~\eqref{eq:bound:uell:prep} and~\eqref{eq:cestimate4} for~$\u_{\ell_1}$ and~$\v_{\ell_2}$ with~$\binom{\ell_2}{2}={\binom{\ell_1}{2}-(\ell_1-\ell_2)(\ell_1+\ell_2-1)/2}$ as well as~$b_0 = {b_1 \cdot p_2^{-p}(1-p_2)^{-(1-p)}}$ and the identity~\eqref{def:x0}, 
in view of~$\ell_2=O(n)$ and $\ell_1 = \Omega(n)$ it follows that 
\begin{equation*}
\begin{split}
    \u_{\ell_1} \v_{\ell_2}&\le \Bigpar{\lrpar{\tfrac{\ell_1}{eN}}^{\ell_1}  b_1^{\binom{\ell_1}{2}}}
 \cdot \Bigpar{ N^{-\ell_2} (p_2^p(1-p_2)^{1-p})^{-\binom{\ell_2}{2}}} \cdot e^{O\lrpar{\log n + n^{3/4}(n-\ell_2)}} \\
    &\le \lrsqpar{ \tfrac{\ell_1}{eN^2} b_0^{(\ell_1-1)/2}  }^{\ell_1} \cdot \lrsqpar{N(p_2^p(1-p_2)^{1-p})^{(\ell_1+\ell_2-1)/2}}^{\ell_1-\ell_2} \cdot e^{O\lrpar{\log n +n^{3/4}(n-\ell_1) + n^{3/4}(\ell_1-\ell_2)}}  \\
    &\leq \lrsqpar{e^{O(\log n/n)+o(n-\ell_1)}\tfrac{\ell_1}{x^{(0)}_N} b_0^{-(x^{(0)}_N-\ell_1)/2}}^{\ell_1} \cdot e^{-\Omega(n(\ell_1-\ell_2))} .
\end{split}
\end{equation*}
Recall that by Lemma~\ref{lem:analytic}~\ref{item:bounded} we have~$b_0=b_0(p)\ge \lambda$ for some constant~${\lambda=\lambda(p_1,p_2)>1}$.
Using~${\zeta n \le \ell_1 \le n} \le {n_N-\eps_N} \le {x^{(0)}_N-\eps_N}$ 
(which implies~$|n-\ell_1| \le x^{(0)}_N-\ell_1$ and~$\ell_1 \le x^{(0)}_N$ as well as~$x^{(0)}_N-\ell_1 \ge \eps_N$ and~$\ell_1 \ge \zeta n$) it follows~that 
\begin{equation}\label{eq:bound:uv:4}
\begin{split}
    \u_{\ell_1} \v_{\ell_2}&\le \lrsqpar{e^{O(\log n/n)} \lambda^{-(x^{(0)}_N-\ell_1)/3}}^{\ell_1} \le 
    e^{-\Omega(\eps_Nn)}
\le n^{-\Omega(\log n)}.
\end{split}
\end{equation}

To sum up: in the case~$\ell_1\ge \max\{\ell_2,2\}$ we each time obtained~$w_{\ell_1,\ell_2}/\mu^2 \le  n^{-\Omega(\log n)}$. 
The same argument (with the roles of~$G_{N,p_1}$ and~$G_{N,p_2}$ interchanged in the definition~\eqref{eq:VarX:well} of~$u_{\ell_1}$ and~$v_{\ell_2}$) also gives~$w_{\ell_1,\ell_2}/\mu^2 \le  n^{-\Omega(\log n)}$ when~$\ell_2\ge \max\{\ell_1,2\}$. 
Inserting these estimates into~\eqref{eq:VarX} readily implies~$\Var X = o((\E X)^2)$.
As discussed, this completes the proof of the lower bound in~\eqref{eq:thmmain}, i.e., that whp~$I_N \ge  \floor{\nN-\eps_N}$, 
which together with the upper bound from  Section~\ref{sec:commonfirst} 
completes the proof of~Theorem~\ref{thm:maincommon}.\noproof

\begin{remark}\label{rem:sec:common}
From Theorem~\ref{thm:maincommon} it readily follows that there is a set~$M\subseteq \bN$ with density~$1$ 
such that, for any $N\in M$, we have~$\floor{n_N-\eps_N}=\floor{n_N+\eps_N}$ and thus whp~$I_N =\floor{n_N-\eps_N}$. 
We leave it as an interesting open problem to show, that for infinitely many~$N$, the size~$I_N$ can be equal to the each of the two different numbers~$\floor{n_N-\eps_N}$ and~$\floor{n_N+\eps_N}$ with probabilities bounded away from~$0$ (similar as for cliques in~$G_{N,p}$; see~\cite[Theorem~11.7]{BB}). 
\end{remark}

\section{Concluding remarks}\label{sec:final}
%
%
In this paper we resolved the induced subgraph isomorphism problem and the maximum common induced subgraph problem for dense random graphs, i.e., with constant edge-probabilities. 
Besides the convergence questions already mentioned in Sections~\ref{sec:contain:dist}, \ref{sec:contain:dist:poisson} and~\ref{sec:common} (see Remarks~\ref{rem:sec:contain:dist}, \ref{rem:sec:dist:poisson} and~\ref{rem:sec:common}),  
there are two main directions for further research: \linebreak[3]
extensions of our main results to sparse random graphs, where edge-probabilities~$p_i \to 0$ are allowed \mbox{(Problem~\ref{prb:sparse})}, 
and generalizations of the graph-sizes \mbox{(Problem~\ref{prb:general})}.   

\begin{problem}[Edge-Sparsity]\label{prb:sparse}
Prove versions of the induced subgraph isomorphism results Theorems~\ref{thm:contain}--\ref{thm:contain:dist} and maximum common induced subgraph result Theorem~\ref{thm:maincommon} for sparse random graphs. 
\end{problem}
As a first step towards such sparse extensions of our main results, one can initially aim at slightly weaker results in the sparse case.  
For example, in Theorem~\ref{thm:maincommon} one could instead try to show that typically~$I_N = (1+o(1))\Lambda_N$ for some explicit~$\Lambda_N=\Lambda_N(p_1,p_2)$.\
Furthermore, in Theorem~\ref{thm:contain} one could instead try to determine some explicit~$n^*=n^*(N,p_1,p_2)$ such that~$\Pr(G_{n,p_1} \isub G_{N,p_2})$ changes from~$1-o(1)$ to~$o(1)$ for~$n \le(1-\eps)n^*$ and~$n \ge (1+\eps) n^*$, respectively.
As another example, in Theorem~\ref{thm:maincommon} with~$p_1=p_2=p$ and~$p=p(N) \to 0$ we wonder if two-point concentration of~$I_N$ remains valid for~$p \gg N^{-2/3+\eps}$.

\begin{problem}[Generalization]\label{prb:general}
Fix constants~$p_1,p_2 \in (0,1)$. 
Determine, as~$N_1,N_2 \to \infty$, the typical size~$I_{N_1,N_2}$ of the maximum common induced subgraph of the independent random graphs~$G_{N_1,p_1}$ and~$G_{N_2,p_2}$.
\end{problem}
This problem aims at a common generalization of our main results, 
since for~$N_1 \le N_2$ the two extreme cases $N_1 \le 2\log_a N_2-\omega(1)$ and~$N_1=N_2$ 
are already covered by Theorems~\ref{thm:contain} and~\ref{thm:maincommon}. 
So the real question is what happens in-between: 
will, similar as in this paper, new concentration phenomena~occur?

\bigskip{\noindent\bf Acknowledgements.} 
We are grateful to the referees for useful suggestions concerning the presentation. 
We also would like to thank Svante Janson for suggesting the inclusion of Corollary~\ref{cor:contain}.

\appendix

\section{Locating the parameter~$n_N$ from Theorem~\ref{thm:maincommon}}
\label{sec:calcstuff}
In this appendix we approximately determine the value of the parameter~$n_N$ from Theorem~\ref{thm:maincommon}, which in~\eqref{def:nN} of Section~\ref{sec:introcommon} is defined as the solution to an optimization problem over all edge-densities~$p \in [0,1]$. 
In particular, 
Lemma~\ref{lem:figcaption} below  locates~$n_N$ in terms of the unique minimizer~$p_0$ of the auxiliary function 
\begin{equation}\label{def:gp:2}
g(p)=g_{p_1,p_2}(p):=\max\bigcpar{\log b_0(p), \: 2\log b_1(p), \: 2\log b_2(p)} \qquad \text{for~$p \in [0,1]$,}
\end{equation}
where the functions~$b_0(p)$ and~$b_i(p)$ that depend on~$p_1,p_2$ are defined as in~\eqref{def:b}. 
Using the standard convention that~$0 \log 0 = \lim_{x \searrow 0} x\log x = 0$ (which is consistent with~$0^0=1$), for~$i \in \{1,2\}$ we here continuously extend $\log b_i(p)$ to~$p \in \{0,1\}$, as usual. 
As in Section~\ref{sec:introcommon}. we also introduce the parameter
\begin{equation}\label{def:phat:2}
\ps = \ps(p_1,p_2) := \frac{p_1p_2}{p_1p_2+(1-p_1)(1-p_2)} 
\end{equation}
that occurs in the different cases of Lemma~\ref{lem:figcaption}, which each assume information about the form of~$g(\ps)$. 
In~\eqref{eq:nN:x0}--\eqref{eq:nN:x1} below we use $a_N \sim b_N$ as a shorthand for~$a_N=(1+o(1))b_N$ as~$N \to \infty$, as~usual.
\begin{lemma}[Locating the parameter~$n_N$]\label{lem:figcaption}
Fix~$p_1,p_2 \in (0,1)$. 
Then the function~$g=g_{p_1,p_2}:[0,1] \to (0,\infty)$ from~\eqref{def:gp:2} 
 is strictly convex, and 
has a unique minimizer~$p_0 \in (0,1)$.
Furthermore, writing~$\ps=\ps(p_1,p_2)$ as in~\eqref{def:phat:2}, the following holds for~$p_0$ and the parameter~$n_N$ from Theorem~\ref{thm:maincommon} defined in~\eqref{def:nN}: 
\begin{romenumerate}
    \item\label{enum:a}  If $\log b_0(\hat{p}) > \max\bigcpar{2\log b_1(\hat{p}), \: 2\log b_2(\hat{p})}$, then $p_0 = \ps$. 
    Moreover, for all large enough~$N$, we~have
\begin{equation}\label{eq:nN:x0}
n_N=\xot_N(\ps) \sim 4 \log_{b_0(\ps)} N .
\end{equation}
    \item\label{enum:b} If $\log b_0(\ps)\le 2\log b_i(\ps)$ for~$i \in \{1,2\}$, then~$p_0$ is the unique solution of~$\log b_0(p)=2\log b_i(p)$, and~$p_0$ lies between~$\ps$ and~$p_i$, with~$p_0 \neq p_i$.  
    Moreover, for all large enough~$N$, we~have
\begin{equation}\label{eq:nN:x1}
n_N=\xonetwo_N(p_0) +O(\log \log N) 
\sim 2 \log_{b_i(p_0)} N .
\end{equation}%
\vspace{-0.25em}\end{romenumerate}\vspace{-0.25em}%
\end{lemma}\vspace{-0.25em}
The proof of Lemma~\ref{lem:figcaption} is based on standard calculus techniques (mainly using convexity), and is spread across the following subsections.
As a byproduct of these techniques, in Appendix~\ref{sec:deferred1} and~\ref{sec:keylemma:apx}
we also give the deferred proofs of the closely related results Corollary~\ref{cor:equalprob}, Remark~\ref{rem:savingold} and Lemma~\ref{lem:analytic}.

\subsection{Proofs of Remark~\ref{rem:savingold} and  Lemma~\ref{lem:analytic}} 
\label{sec:deferred1}
%
The functions $g(p)$ and ${\log b_j(p)}$ appearing in the following auxiliary lemma of course depend on~$p_1,p_2 \in (0,1)$, but to avoid clutter we henceforth suppress the dependence on~$p_1,p_2$ in the notation, as~usual.
\begin{lemma}\label{lem:strictlyconvex}
Fix~$p_1,p_2 \in (0,1)$. 
The functions $\log b_0(p)$, $\log b_1(p)$ and~$\log b_2(p)$ are strictly convex functions for $p\in [0,1]$, and achieve their unique minima in $[0,1]$ at $\ps$, $p_1$ and~$p_2$, respectively.
Furthermore, 
the function~$g:[0,1] \to (0,\infty)$ is strictly convex, and has a unique minimizer in~$[0,1]$ at~$p_0 \in (0,1)$.
\end{lemma}
\begin{proof}
We start with the functions $\log b_j(p)$. 
By routine calculus, the first and second derivatives are 
\begin{equation}\label{eq:firstderivatives}
\begin{split}
    \frac{\partial}{\partial p} \log b_0(p) & =\log\lrpar{ \frac{p(1-p_1)(1-p_2)}{(1-p)p_1p_2}},\\
    \frac{\partial}{\partial p} \log b_i(p)& =\log\lrpar{ \frac{p(1-p_i)}{(1-p)p_i}} \phantom{\frac{1}{p(1-p)}} \quad \text{for~$i \in \{1,2\}$,}\\ %
    \frac{\partial^2}{\partial^2 p} \log b_j(p) & 
    = \frac{1}{p(1-p)} \phantom{\log\lrpar{ \frac{p(1-p_i)}{(1-p)p_i}}} \quad \text{for~$j \in \{0,1,2\}$.} %
\end{split}
\end{equation}
Note that the $\log b_j(p)$ with~$j \in \{0,1,2\}$ are all strictly convex functions for~$p\in [0,1]$, 
since each function is continuous on $[0,1]$ with a strictly positive second derivative for $p \in (0,1)$.
Note that the first derivatives~\eqref{eq:firstderivatives} of the~$\log b_j(p)$ are all negative near~$0$ and positive near~$1$. 
So by solving for the zeroes of the first derivative,  
it follows that $\log b_0(p)$, $\log b_1(p)$ and~$\log b_2(p)$ achieve their unique minima at $\ps$, $p_1$ and~$p_2$, respectively.

We now turn to the function~$g(p)$.
Letting $f_0(x) = \log b_0(x)$, $f_1(x) = 2\log b_1(x)$ and $f_2(x) = 2\log b_2(x)$, 
note that for any $a,b \in [0,1]$ and~$t \in (0,1)$ with~$a \neq b$, 
by strict convexity of the functions~$f_j(x)$ it follows~that
\begin{align*}
    g(t a+(1-t)b)&=\max_{j \in \{0,1,2\}} \{f_j(t a+(1-t)b)\}
    < \max_{j \in \{0,1,2\}} \{t f_j(a)+(1-t)f_j(b)\}
    \le t g(a)+(1-t) g(b),
\end{align*}
which establishes that~$g(p)$ is a strictly convex function for~$p\in [0,1]$.
As uniqueness of the minimizer~$p_0$ of~$g(p)$ over~$p \in [0,1]$ follows from strict convexity, it suffices to check that the minimum of $g(p)$ is not attained at the endpoints~${p \in \{0,1\}}$. 
This follows from the behavior of the first derivatives~\eqref{eq:firstderivatives} of the~$\log b_j(p)$ established above, which imply that~$g(p)$ is decreasing near~$0$ and increasing near~$1$. 
Finally, it remains to determine the range of~$g(p)$ for all~$p \in [0,1]$: the upper bound $g(p) \le \max\{g(0),g(1)\} <\infty$ follows from strict convexity of~$g(p)$ and the convention~$0\log 0=0$  (mentioned at the beginning of Appendix~\ref{sec:calcstuff}), 
and the lower bound $g(p)\ge \log b_0(p)\ge \log b_0(\ps) >0$ follows from 
the properties of~$\log b_0(p)$ established above.
\end{proof}

\begin{proof}[Proof of Lemma~\ref{lem:analytic}]
\ref{item:approximate}: 
From Lemma~\ref{lem:strictlyconvex} we know that $g:[0,1] \to (0,\infty)$ has a unique minimizer~$p_0 \in (0,1)$, 
and that~$g(p)$ is strictly convex, which together readily establishes Lemma~\ref{lem:analytic}~\ref{item:approximate}. 

\ref{item:bounded}: 
From Lemma~\ref{lem:strictlyconvex}, for any $p\in [0,1]$ we know that $b_0(p)\ge b_0(\ps)>1$ and $b_i(p)\ge b_i(p_i)\ge 1$ for~$i \in \{1,2\}$. 
This enables us to pick~$\eta=\eta(p_1,p_2)>0$ sufficiently small such that, for all sufficiently large~$N$, we~have
\begin{equation} \label{def:lambda}
(4\log_{b_0(\ps)}N+1) \cdot  (1+\eta)^{4\log_{b_0(\ps)}N} \le N \qquad \text{and} \qquad \log b_0(\ps) \ge 3\log (1+\eta).
\end{equation}
Setting~$\lambda := 1+\eta$, we first analyze properties of~$\xot_N(p)$. 
Since~$\log b_0(p) \ge \log b_0(\ps) >  \log \lambda > 0$, 
by inspecting the definition~\eqref{def:x_N} of $\xot_N(p)$ we infer that, for all sufficiently large~$N$ (depending only on~$\lambda$), we have
\begin{align} \label{eq:xot:bound}
\xot_N(p) & \le 4\log_{b_0(\ps)}N+1,\\
\label{eq:xot:asymp}
\xot_N(p) & = 4\log_{b_0(p)} N + O(\log \log N) , 
\end{align} 
where the asymptotic estimate in~\eqref{eq:xot:asymp} is uniform in~$p$, i.e., the implicit error term~$O(\log \log N)$ depends only on~$p_1,p_2$ (and not on~$p$). 
%
We next turn to~$x^{(i)}_N(p)$ with~$i\in \{1,2\}$, 
where ${b_i=b_i(p) \ge 1}$ holds.
Using the definition~\eqref{def:xi} it is easy to see that~$x^{(i)}_N(p) > 1$ holds for all~$N \ge 1$ (since~$x^{(i)}_N(p) \le 1$ implies that $e \le eN \le 1 \cdot b_i^{-(1-x^{(i)}_N(p))/2} \le 1$).
For technical reasons we now distinguish whether~${b_i=b_i(p)}$ is smaller or larger than~$\lambda$. 
We start with the case~${1 \le b_i < \lambda}$, 
where using~\eqref{def:xi} and~\eqref{def:lambda} we infer~that
\[ 
x^{(i)}_N(p) \cdot \lambda ^{(x^{(i)}_N(p)-1)/2} 
\ge 
eN>(4\log_{b_0(\ps)}N+1) \cdot  \lambda^{[(4\log_{b_0(\ps)}N+1)-1]/2} ,\]
which in turn, by noting that~$x\lambda^{(x-1)/2}$ is increasing for~$x \ge 1$, implies together with~\eqref{eq:xot:bound} that
\begin{equation} \label{eq:xiot:comparison}
x^{(i)}_N(p)> 4\log_{b_0(\ps)}N+1 \ge x^{(0)}_N(p) .
\end{equation}
In this case, gearing up towards~\eqref{eq:firstorderofmin}, using~\eqref{def:lambda} we also infer that 
\begin{equation} \label{eq:xiot:comparison:firstorder}
\frac{2\log N}{\log b_i(p)}\ge \frac{2\log N}{\log \lambda} >\frac{4\log N}{\log b_0(p)}.
\end{equation}
We next consider the case~$b_i\ge \lambda > 1$.
Applying a bootstrapping argument to the implicit definition~\eqref{def:xi} of $x^{(i)}_N(p)$, i.e., $eN=x^{(i)}_N(p) b_i^{(x^{(i)}_N(p)-1)/2}$, 
for all sufficiently large~$N$ (depending only on~$\lambda$) 
it follows that
\begin{equation}\label{eq:xifirstorder}
\begin{split}
     x^{(i)}_N(p)&=2\log_{b_i}N -2\log_{b_i} x^{(i)}_N(p)+2\log_{b_i} e +1\\
    &= 2\log_{b_i}N -2\log_{b_i}(2\log_{b_i}N -2\log_{b_i} x^{(i)}_N(p)+2\log_{b_i} e +1)+2\log_{b_i} (e)+1
        \\
    &= 2\log_{b_i}N -2\log_{b_i}\log_{b_i}N-2\log_{b_i} (2/e)+1 +O\lrpar{\frac{\log\log N}{\log N}} \\
    & = 2\log_{b_i}N + O(\log \log N) , 
\end{split}
\end{equation}
where all the asymptotic estimates in~\eqref{eq:xifirstorder} are uniform, i.e., do not depend on~$p$ (here we exploit that~${b_i \ge \lambda > 1}$ in this case).
In both cases,
by combining~\eqref{eq:xot:asymp} with~\eqref{eq:xiot:comparison}--\eqref{eq:xiot:comparison:firstorder} and~\eqref{eq:xifirstorder}, we obtain the uniform estimate 
\begin{equation}
\label{eq:xiot:comparison:firstorder:asymp}
\begin{split}
\min \bigcpar{\xot_N(p),\: x^{(i)}_N(p)} 
& = \min \left\{ \frac{4\log N}{\log b_0(p)},\: \frac{2\log N}{\log b_i(p)} \right\} + O(\log\log N) .
\end{split}
\end{equation}
Applying~\eqref{eq:xiot:comparison:firstorder:asymp} for both~$i \in \{1,2\}$ then establishes the desired  estimate~\eqref{eq:firstorderofmin} by definition of~$g(p)$, 
completing the proof of Lemma~\ref{lem:analytic}~\ref{item:bounded}. 
\end{proof}

\begin{proof}[Proof of Remark~\ref{rem:savingold}]
The definition of~$n_N$ is exactly as in the proof of Theorem~\ref{thm:maincommon}, 
so we only need to establish the asymptotic estimate~\eqref{def:y_N:asymp}.
Here the crux is that from~\eqref{eq:xiot:comparison} we know that $x^{(i)}_N(p) \leq x^{(0)}_N(p)$ implies~$b_i \ge \lambda$, 
so that estimate~\eqref{eq:xifirstorder} applies to~$x^{(i)}_N(p)$, 
which immediately establishes~\eqref{def:y_N:asymp}, as desired. 
\end{proof}

\subsection{Proofs of Corollary~\ref{cor:equalprob} and Lemma~\ref{lem:figcaption}}\label{sec:keylemma:apx}
\begin{proof}[Proof of Corollary~\ref{cor:equalprob}]
By assumption, Lemma~\ref{lem:figcaption}~\ref{enum:a} implies that $n_N=\xot_N(\ps)$ for all~$N$ large enough, 
Invoking Theorem~\ref{thm:maincommon}, it thus only remains to verify that~$\log b_0(\ps) > \max\{2\log b_1(\ps),2\log b_2(\ps)\}$ holds when~$p_1=p_2=p$. 
To this end, we start with the trivial inequality 
\[ \frac{p^4+(1-p)^4}{p^2+(1-p)^2}<p^2+(1-p)^2.\]
Since $\log(x)$ is a concave function, using~$\ps=\ps(p,p)=p^2/(p^2+(1-p)^2)$ and Jensen's inequality it follows~that
\[ \ps \log p^2 +(1-\ps) \log (1-p)^2 \le \log\lrpar{\frac{p^4+(1-p)^4}{p^2+(1-p)^2}} < \log\bigpar{p^2+(1-p)^2}.\]
Exponentiating both sides, we infer that
\[ p^{2\ps} (1-p)^{2(1-\ps)}< p^2+(1-p)^2.\]
Dividing both sides by~$(p^2+(1-p)^2)^2>0$, using~$p_1=p_2=p$ we conclude that
\[b_1(\ps)^2=b_2(\ps)^2= \lrpar{\frac{\ps}{p}}^{2\ps} \lrpar{\frac{1-\ps}{1-p}}^{2(1-\ps)}=\frac{p^{2\ps}(1-p)^{2(1-\ps)}}{(p^2+(1-p)^2)^2}< \frac{1}{p^2+(1-p)^2}=b_0(\ps), \]
which establishes that~$\log b_0(\ps) > \max\{2\log b_1(\ps),2\log b_2(\ps)\}$, 
as desired.
\end{proof}

\begin{proof}[Proof of Lemma~\ref{lem:figcaption}]
\ref{enum:a}: From Lemma~\ref{lem:strictlyconvex}, we know that $\log b_0(p)$ achieves its unique minimum at~$p=\ps$.
Our assumption implies that $g(\ps) = \log b_0(\ps)$. 
For any $p\in[0,1]$ we thus~have
\begin{equation*}
    g(p) \ge \log b_0(p)\ge \log b_0(\ps)= g(\ps),
\end{equation*}  
which establishes that $\ps$ is the unique (see Lemma~\ref{lem:strictlyconvex}) minimizer~$p_0$ of the function~$g(p)$.

It remains to prove that $n_N=\xot_N(\ps)$ for all large enough $N$. 
Our assumption implies $\tfrac{1}{4}
\log b_0(\ps)>\max\{\tfrac{1}{2}\log b_1(\ps), \tfrac{1}{2}\log b_2(\ps)\}$, 
which by comparing the first order terms in~\eqref{eq:xiot:comparison:firstorder:asymp}, \eqref{eq:xiot:comparison} and \eqref{eq:xifirstorder} implies 
\begin{equation}\label{eq:order:xot}
\xot_N(\ps)<\min\left\{\xone_N(\ps), \: \xtwo_N(\ps)\right\}
\end{equation}
for all~$N$ large enough (depending only on the functions~$\log b_j(\ps)$ and thus~$p_1,p_2$). 
Fix~$p\in [0,1]$ with~$p \neq \ps$. Setting~$r:=\log b_0(\ps) > 0$ and $\gamma:=\log b_0(p)-r >0$ 
(where~$\gamma > 0$ follows from~$p \neq \ps$, as $\ps$ is the unique minimizer of $\log b_0(p)$ by Lemma~\ref{lem:strictlyconvex}), 
using~${\log(1+\gamma/r) \le \gamma/r}$ it follows for all~$N$ large enough (depending only on~$\log b_0(\ps)$ and thus~$p_1,p_2$) that 
\begin{equation}\label{eq:order:xot:diff}
\begin{split}
        &\xot_N(\ps)-\xot_N(p)\\
        & = \frac{4\log N-2\log\log N-2\log(4/e)+2\log r}{r} - \frac{4\log N-2\log\log N-2\log(4/e)+2\log (r+\gamma)}{r+\gamma}\\
        &=\frac{\gamma}{r(r+\gamma)} \biggsqpar{4\log N-2\log\log N -2\log(4/e)-\frac{2r}{\gamma}\log\lrpar{1+\frac{\gamma}{r}} + 2\log r}\\
        & \ge \frac{\gamma}{r(r+\gamma)} \Bigsqpar{2\log N-2+2 \log \log b_0(\ps)} \ge \frac{\gamma \log N}{r(r+\gamma)} > 0 .
\end{split}
\end{equation}
Combining inequalities~\eqref{eq:order:xot} and~\eqref{eq:order:xot:diff}
with the definition~\eqref{def:nN} of~$n_N$, it then follows that
\[
\xot_N(\ps) = \min\left\{\xot_N(\ps), \: \xone_N(\ps), \: \xtwo_N(\ps) \right\} \le n_N \le \max_{p\in [0,1]} \xot_N(p) = \xot_N(\ps) .
\]
Hence~$n_N=\xot_N(\ps)$ for all large enough~$N$ (depending only on~$p_1,p_2$), 
which similar to the asymptotic reasoning for~\eqref{eq:xot:bound} implies~\eqref{eq:nN:x0} by the definition of $\xot_N(\ps)$, 
and thus completes the proof of Lemma~\ref{lem:figcaption}~\ref{enum:a}. 
    
\ref{enum:b}: 
First, we show that the solution~$\pz$ of~$\log b_0(p)=2\log b_i(p)$ is unique, and that~$\pz$ lies between~$\ps$ and~$p_i$, with~$\pz \neq p_i$. 
To this end we introduce the auxiliary function
\begin{align*}
    h(p):=2\log b_i(p)-\log {b_0}(p)= p \log \lrpar{\frac{pp_1p_2}{p_i^2}}+ (1-p)\log\lrpar{\frac{(1-p)(1-p_1)(1-p_2)}{(1-p_i)^2}}.
\end{align*}
Using~\eqref{eq:firstderivatives} we see that the second derivative is $h''(p)=1/[p(1-p)]$. 
Hence~$h(p)$ is a strictly convex function for~$p\in [0,1]$, 
since it is a continuous function on $[0,1]$ with a strictly positive second derivative for $p \in (0,1)$.
The above proof of Corollary~\ref{cor:equalprob} 
shows that $p_1=p_2$ falls in the previous case~\ref{enum:a}, so that we here have~$p_1 \neq p_2$. 
Consequently, $h(0)$ and $h(1)$ are both not equal to~$0$, and 
in particular have different signs (using~$0\log 0=0$, as mentioned at the beginning of Appendix~\ref{sec:calcstuff}). 
Using strict convexity it follows that~$h(p)$ has a unique zero in~$[0,1]$, which by construction is the unique solution~$\pz$ of~$\log b_0(p)=2\log b_i(p)$. 
Furthermore, since our assumptions imply $2\log b_i(\ps) \ge \log b_0(\ps)$ and~$p_1,p_2 \in (0,1)$, for~$j \in \{1,2\} \setminus \{i\}$ it follows that
\begin{equation*}
h(\ps)\ge 0 \qquad\text{and} \qquad h(p_i)=p_i\log p_j+(1-p_i) \log (1-p_j)<0.
\end{equation*}
which establishes that~$\pz$ lies between $\ps$ and $p_i$, with~$\pz \neq p_i$. 

Second, we claim that $f(p):=\max\{\log b_0(p), 2\log b_i(p)\}$ is decreasing in~$[0,\pz]$ and increasing in~$[\pz,1]$, which together with~$\log b_0(\pz)=2\log b_i(\pz)$ establishes the useful fact 
\begin{equation}\label{eq:f:optimizer}
\min_{p\in [0,1]}\max\bigcpar{\log b_0(p),\: 2\log b_i(p)}= 2\log b_i(\pz).
\end{equation}
By Lemma~\ref{lem:strictlyconvex}, $\log b_0(p)$ is decreasing in~$[0,\ps]$ and increasing in~$[\ps,1]$. 
Furthermore, $2\log b_i(p)$ is decreasing in~$[0,p_1]$ and increasing in~$[p_1,1]$. 
Hence~$f(p)$ is  increasing in~$[\max\{p_1,\ps\}, 1]$.
Furthermore, recalling that~$\pz$ lies between~$\ps$ and~$p_1$, it follows that in~$[\pz, \max\{p_1,\ps\}]$  one of~$\log b_0(p)$ or~$2\log b_i(p)$ is increasing, whereas the other one is decreasing. 
Since~$\log b_0(\pz) = 2\log b_i(\pz)$, it follows that $f(p)$ is increasing in~$[\pz, \max\{p_1,\ps\}]$, 
establishing that $f(p)$ is  increasing in~$[\pz,1]$. 
A similar argument shows that $f(p)$ is decreasing in~$[0,\pz]$.  

Third, we show that $\pz$ is the minimizer of~$g(p)$.
To see this, note that~$\log b_0(\pz)=2\log b_i(\pz)$ implies
\[ b_0(\pz) = b_0(\pz)  \cdot \frac{b_0(\pz)}{b_i(\pz)^2} 
=\lrsqpar{\lrpar{\frac{1}{p_j}}^p \lrpar{\frac{1}{ 1-p_j}}^{1-p}}^2 \ge b_j(\pz)^2  \]
for~$j \in \{1,2\} \setminus \{i\}$,
so that~$g(\pz)=f(\pz) = 2\log b_i(\pz)$. 
Using~\eqref{eq:f:optimizer}, for any~$p\in [0,1]$ it then follows~that
\[ g(p)\ge \max\bigcpar{\log b_0(p),\: 2\log b_i(p)} \ge  2\log b_i(\pz)=  g(\pz),\]
which establishes that $\pz$ is the unique (see Lemma~\ref{lem:strictlyconvex}) minimizer~$p_0$ of the function~$g(p)$. 

Finally, from Lemma~\ref{lem:analytic}~\ref{item:bounded} and~\eqref{eq:firstorderonN}, 
using~$p_0=\pz \in (0,1)$ and~$g(p_0)=2\log b_i(p_0)>0$ it follows that 
\[ n_N= \frac{4\log N}{g(p_0)} + O(\log\log N)= \xonetwo_N(p_0) + O(\log\log N) =(1+o(1)) 2\log_{b_i(p_0)}N , \]
which establishes~\eqref{eq:nN:x1} and thus 
completes the proof of Lemma~\ref{lem:figcaption}~\ref{enum:b}.
\end{proof}

\section{Proof of \refL{lem:typical}: Pseudorandom properties}\label{apx:pseudorandom}
\begin{proof}[Proof of \refL{lem:typical}]
We start with the binomial random~$G_{n,p}$, 
whose number of edges has a binomial distribution with expected value~$\tbinom{n}{2} \cdot p$. 
Using Chebyshev's inequality (or Chernoff bounds),
it easily follows~that 
\begin{equation}\label{eq:edges:gnp}
\Pr(G_{n,p} \not\in \cF_{n,p}) 
= o(1). 
\end{equation}
Furthermore, the estimate 
\begin{equation}\label{eq:asym:gnp:basic}
\Pr(G_{n,p} \text{ is not asymmetric}) \le e^{-\Theta(np(1-p))} 
\end{equation}
follows, for instance, by the proof of Theorem 3.1 in~\cite{kim2002asymmetry} when replacing the~$\eps$ in their proof by some small constant (their proof in fact works for any~$p=p(n)$ satisfying $np(1-p) \gg \log n$).  Note that the induced subgraph~$G_{n,p}[L]$ has the same distribution as~$G_{|L|,p}$. 
Taking a union bound over all possible vertex-subsets~$L \subseteq [n]$ of size~$\ell := |L| \ge n-n^{2/3}$,   
using~\eqref{eq:asym:gnp:basic} and~$\tbinom{n}{\ell}=\tbinom{n}{n-\ell} \le n^{n-\ell} \le n^{n^{2/3}}$ we~obtain 
\begin{equation}\label{eq:asym:gnp}
\Pr(G_{n,p} \not\in \cA_n) \le \sum_{n-n^{2/3} \le \ell \le n} \binom{n}{\ell}  \cdot e^{-\Theta(\ell p(1-p))} \le e^{-\Theta(np(1-p))} = o(1), 
\end{equation}
which in fact holds for any edge-probability~$p=p(n)$ satisfying $n^{1/3}p(1-p) \gg \log n$. 
Combining~\eqref{eq:edges:gnp} and~\eqref{eq:asym:gnp} establishes that~$G_{n,p}\in \cA_n \cap \cF_{n,p}$ whp. 

We next consider the  uniform random graph~$G_{n,m}$. Using Pittel's inequality (see~\cite[Theorem~2.2]{BB}) with~$p'=p'(n):=m/\tbinom{n}{2} \in [\gamma,1-\gamma]$, it routinely follows from inequality~\eqref{eq:asym:gnp} that
\begin{equation}\label{eq:asym:gnm}
\Pr(G_{n,m} \not\in \cA_n) \le 3 \sqrt{m} \cdot \Pr(G_{n,p'} \not\in \cA_n) \le O(n) \cdot  e^{-\Theta(np'(1-p'))}  = o(1). 
\end{equation}
Turning to~$\cE_{n,m}$, 
fix any~$L \subseteq [n]$ and write~$\ell := |L|$ for the size of~$L$, as before. 
To estimate the number of edges contained in the induced subgraph~$G_{n,m}[L]$, 
we use that by definition of~$G_{n,m}$ we have
\begin{equation}\label{eq:edge-decomp}
\underbrace{|E(G_{n,m})|}_{=m}
= |E(G_{n,m}[L])| + \underbrace{|E(G_{n,m}) \setminus E(G_{n,m}[L])|}_{=: X_L}.
\end{equation}
For~$L=[n]$ this already  gives~$|E(G_{n,m}[L])|=\tbinom{|L|}{2}m/\tbinom{n}{2}$, 
so it remains to consider the case~${L \subsetneq [n]}$.
The random variable~$X_L$ has a hypergeometric distribution 
with parameters~$K:=\tbinom{n}{2}$, $m= \Theta(n^2)$ and 
\begin{equation*}
k_{\ell}:=\binom{n}{2}-\binom{\ell}{2} = \frac{(n-\ell)(n+\ell-1)}{2},
\end{equation*}
with expected value $\E X_L=k_{\ell} \cdot m/K = \Theta(n(n-\ell))$, 
since we have~$k_{\ell}$ random draws (without replacement) out of a total of $K=\tbinom{n}{2}$ potential edges, 
where~$m \in {[\gamma \tbinom{n}{2},(1-\gamma)\tbinom{n}{2}]}$ out of the potential edges are present. 
Since standard Chernoff bounds for binomial random variables with expected value~$k_{\ell} \cdot m/K$
also apply to~$X_L$, see~\mbox{\cite[Theorem~2.1 and~2.10]{JLR}},
it routinely follows that 
\[
\Pr\bigpar{|X_L-k_{\ell} m/K| \ge n^{2/3}(n-\ell)} \le 2 \exp\lrpar{\frac{-n^{4/3}(n-\ell)^2}{3 \E X_L}} \le e^{-\Theta(n^{1/3}(n-\ell))} . 
\]
Writing~$\cE'_{n,m}$ for the event that~$|X_L-k_{\ell} m/K| < n^{2/3}(n-\ell)$ for all~$L \subsetneq [n]$, using a standard union bound argument and $\tbinom{n}{\ell}=\tbinom{n}{n-\ell}\le n^{n-\ell}$ it readily follows that 
\begin{equation}\label{eq:asym:Bnm}
\Pr(G_{n,m} \not\in \cE'_{n,m}) \le \sum_{0 \le \ell \le n-1}\binom{n}{\ell} \cdot e^{-\Theta(n^{1/3}(n-\ell))} = o(1) .
\end{equation}
Combining~\eqref{eq:asym:gnm} and~\eqref{eq:asym:Bnm} establishes that~$G_{n,m} \in \cA_n \cap \cE'_{n,m}$ whp. 
It remains to show that~$G_{n,m} \in \cE'_{n,m}$ implies~$G_{n,m} \in \cE_{n,m}$. 
To see this note that, for any~$L \subsetneq [n]$, equation~\eqref{eq:edge-decomp} and~$G_{n,m} \in \cE'_{n,m}$ imply
\begin{equation*}
|E(G_{n,m}[L])| = m - X_L = m  - \frac{\bigsqpar{\binom{n}{2}-\binom{|L|}{2}}m}{\binom{n}{2}} \pm n^{2/3}(n-|L|) = \frac{\tbinom{|L|}{2}m}{\binom{n}{2}} \pm n^{2/3}(n-|L|),
\end{equation*}
which completes the proof of \refL{lem:typical} (since this edge-estimate is trivially true for~$L=[n]$, as discussed). 
\end{proof}

\end{document}